\documentclass{amsart}
\usepackage{amsmath,amssymb}

\usepackage{amsxtra, amsmath}
\usepackage{amssymb, amscd}
\usepackage{graphicx}

\newcommand\lieg{{\mathfrak g}}
\newcommand\liek{{\mathfrak k}}

\newcommand\liesl{{\mathfrak sl}}

\newcommand\liegl{{\mathfrak gl}}

\newcommand\liem{{\mathfrak m}}

\newcommand\ee{{\mathbf e}}

\newcommand\mm{{\mathbf m}}
\newcommand\kk{{\mathbf k}}
\newcommand\pp{{\mathbf p}}
\newcommand\rr{{\mathbf r}}
\newcommand\CC{\mathbb C}

\newcommand\RR{\mathbb R}
\newcommand\ZZ{\mathbb Z}
\newcommand\NN{\mathbb N}

\newcommand\A{{\mathcal A}}

\newcommand\GL{{\mathrm{GL}}}

\newcommand\SU{{\mathrm{SU}}}
\newcommand\U{{\mathrm{U}}}
\newcommand\adj{{\mathrm{adj}}}

\newcommand\End{\operatorname{End}}
\newcommand\re{\operatorname{Re}}
\newcommand\im{\operatorname{Im}}
\newcommand\tr{\operatorname{tr}}

\makeatletter 
\newcommand\matc[4]{\left( {#1\@@atop #3}{#2\@@atop #4}\right)}
\newcommand\matr[4]{\left( {\hfill #1\@@atop\hfill #3}{\hfill
#2\@@atop\hfill #4}\right)}
\newcommand\matl[4]{\left( { #1\@@atop #3}{ #2\@@atop\hfill #4}\right)}
\makeatother
\newcommand\widearray[1]{\renewcommand\arraystretch{1.4} \begin{array}{#1}}

\newcommand\vzm[2]{\mathchoice{\left\{\, #1 : #2 \,\right\}}{\{\,
#1:\allowbreak #2 \,\}}{\{ #1 :\allowbreak #2 \}}
{\{ #1 :\allowbreak #2 \}}}

\theoremstyle{plain}
\newtheorem{thm}{Theorem}[section]
\newtheorem{lem}[thm]{Lemma}

\newtheorem{prop}[thm]{Proposition}
\newtheorem{cor}[thm]{Corollary}

\theoremstyle{definition}
\newtheorem{defn}[thm]{Definition}

\newtheorem{remark}{Remark}[section]

\newtheorem*{prop3.3}{Proposition \ref{expD1}}
\newtheorem*{prop3.4}{Proposition \ref{expD2}}
\newtheorem*{prop3.5}{Proposition \ref{expE1}}
\newtheorem*{prop3.6}{Proposition \ref{expE2}}

\title[One-step spherical functions  of the pair $(\SU(n+1),\U(n))$ ]{One-step spherical functions of the pair $(\SU(n+1),\U(n))$ }
\author{In\'es Pacharoni}
\author{Juan Tirao}
\address{CIEM-FaMAF, Universidad Nacional de C\'ordoba, 5000 C\'ordoba, Argentina}
\email{pacharon@famaf.unc.edu.ar, tirao@famaf.unc.edu.ar}
\dedicatory{Dedicated to our teacher and friend Joe Wolf.}

\thanks{{\sc  CIEM-FaMAF, Universidad Nacional de C\'ordoba, 5000 C\'ordoba, Argentina} }
\thanks {{\em E-mail address}: \texttt {pacharon\@@famaf.unc.edu.ar, tirao\@@famaf.unc.edu.ar}}

\begin{document}

\begin{abstract}
The aim of this paper is to determine all irreducible spherical
functions of  the pair $(G,K)=(\SU(n+1), \U(n))$, where the highest weight of their  $K$-types
$\pi$ are of the form $(m+\ell,\dots,m+\ell,m,\dots,m)$.
Instead of looking at a  spherical function $\Phi$ of type $\pi$ we look at a matrix-valued function $H$ defined on a section of the $K$-orbits in an affine subvariety of $P_n(\CC)$.
The function  $H$ diagonalizes, hence it can be identified with a column vector-valued function.
The irreducible spherical functions of type $\pi$  turn out to be parameterized by
$S=\{(w,r)\in \ZZ\times\ZZ:\, 0\leq w, \,0 \le r\le\ell, \,0\leq  m+w+r \}$.
A key result to characterize the associated function $H_{w,r}$ is the existence of a matrix-valued polynomial function $\Psi$ of degree $\ell$ such that
$F_{w,r}(t)=\Psi(t)^{-1}H_{w,r}(t)$ becomes an eigenfunction of a matrix hypergeometric operator with eigenvalue $\lambda(w,r)$, explicitly given. In the last section we assume that $m\ge0$ and define the matrix
polynomial $P_w$ as the $(\ell+1)\times(\ell+1)$ matrix whose
$r$-row is the polynomial $F_{w,r}$. This leads to interesting families of matrix-valued orthogonal Jacobi polynomials $P_w^{\alpha,\beta}$ for $\alpha,\beta>-1$.
\\


{\em MSC 2010 Codes:} aaa

\end{abstract}

\maketitle

\section{Spherical functions}

\noindent Let $G$ be a locally compact unimodular group and let $K$ be a compact
subgroup of $G$. Let $\hat K$ denote the set of all
equivalence classes of complex finite-dimensional irreducible
representations of $K$; for  $\delta\in \hat K$, let
$\xi_\delta$ denote the character of $\delta$, $d(\delta)$ the degree of
$\delta$, i.e., the dimension of any representation in
the class $\delta$, and $\chi_\delta=d(\delta)\xi_\delta$. We choose the Haar measure $dk$ on
$K$ normalized by $\int_K dk=1$.
We shall denote by $V$ a finite-dimensional vector space over the field
$\CC$ of complex numbers and by $\End(V)$ the space
of all linear transformations of $V$ into $V$.

A {\em spherical function} $\Phi$ on $G$ of type $\delta\in \hat K$ is a
continuous function on $G$ with values in $\End(V)$ such
that
\begin{enumerate} \item[(i)] $\Phi(e)=I$ ($I$= identity transformation).

\item[(ii)] $\Phi(x)\Phi(y)=\int_K \chi_{\delta}(k^{-1})\Phi(xky)\, dk$, for
all $x,y\in G$.
\end{enumerate}

 \noindent If $\Phi:G\longrightarrow
\End(V)$ is a spherical function of type $\delta$,
then $\Phi(kgk')=\Phi(k)\Phi(g)\Phi(k')$, for all $k,k'\in K$, $g\in G$,
and  $k\mapsto \Phi(k)$ is a representation of $K$ such that any
irreducible subrepresentation belongs to $\delta$. In particular the spherical
function $\Phi$ determines its type univocally and let us say that the number of times that $\delta$
occurs in the representation $k\mapsto \Phi(k)$ is called the {\em
height} of $\Phi$.

Spherical functions of type $\delta$ arise in a natural way upon
considering representations of $G$.
If $g\mapsto U(g)$ is a continuous representation of $G$, say on a finite
dimensional  vector
space $E$, then $$P(\delta)=\int_K \chi_\delta(k^{-1})U(k)\, dk$$ is a
projection of $E$ onto
$P(\delta)E=E(\delta)$; $E(\delta)$ consists of those vectors in $E$, the
linear span of whose $K$-orbit splits into irreducible $K$-subrepresentations
of type $\delta$.
The function
$\Phi:G\longrightarrow \End(E(\delta))$ defined by
$$\Phi(g)a=P(\delta)U(g)a,\quad g\in G,\; a\in E(\delta)$$
is a spherical function of type $\delta$. In fact, if $a\in E(\delta)$ we have
\begin{align*}
\Phi(x)\Phi(y)a&= P(\delta)U(x)P(\delta)U(y)a=\int_K \chi_\delta(k^{-1})
P(\delta)U(x)U(k)U(y)a\, dk\\
&=\left(\int_K\chi_\delta(k^{-1})\Phi(xky)\, dk\right) a. \end{align*}
If the representation $g\mapsto U(g)$ is irreducible, then the associated
spherical function $\Phi$ is also irreducible. Conversely, any irreducible
spherical function on a compact group $G$ arises in this way from a finite-dimensional irreducible representation of $G$.

If $G$ is a connected Lie group, it is
not difficult to prove that any spherical function $\Phi:G\longrightarrow
\End(V)$ is
differentiable ($C^\infty$), and moreover that it is analytic. Let $D(G)$
denote the algebra of all
left-invariant differential operators on $G$ and let $D(G)^K$ denote the
subalgebra of all operators in
$D(G)$ that are invariant under all right
translations by elements in $K$.

\smallskip
In the following proposition $(V,\pi)$ will be a finite-dimensional
representation of $K$ such that any irreducible
subrepresentation belongs to the same class $\delta\in\hat K$.
\begin{prop}\label{defeq}(\cite{T1},\cite{GV}) A function
$\Phi:G\longrightarrow \End(V)$ is a spherical function of type $\delta$ if
and only if
\begin{enumerate}
\item[(i)] $\Phi$ is analytic.
\item[(ii)] $\Phi(kgk')=\pi(k)\Phi(g)\pi(k')$, for all $k,k'\in K$,
$g\in G$, and
$\Phi(e)=I$.
\item[(iii)] $[D\Phi ](g)=\Phi(g)[D\Phi](e)$, for all $D\in D(G)^K$, $g\in G$.
\end{enumerate}
\end{prop}

\smallskip
The aim of this paper is to determine all
irreducible spherical functions of  the
pair $(G,K)=(\SU(n+1), {\mathrm S}(\U(n)\times \U(1)))$, $n\ge2$, whose  $K$ types are of a special kind.
This will be done starting from Proposition \ref{defeq}.

The irreducible finite-dimensional representations of $G$
are restrictions of  irreducible representations of $\U(n+1)$, which
are parameterized by $(n+1)$-tuples of integers
$$\mm=(m_{1}, m_{2},\dots, m_{n+1})$$
such that $m_{1}\geq m_{2}\geq \cdots \geq m_{n+1}$.

Different representations of $\U(n+1)$ can be restricted to the same
representation of $G$. In fact the representations $\mm$
and $\pp$ of $\U(n+1)$ restrict to the same representation of
$\SU(n+1)$ if and only if $m_i=p_i+j$ for all $i=1,\dots,n+1$ and
some $j\in\ZZ$.

The closed subgroup $K$ of $G$ is
isomorphic to $\U(n)$, hence its finite-dimensional irreducible
representations are parameterized by the $n$-tuples of integers
$$\kk=(k_{1}, k_{2},\dots,k_{n})$$ subject to the conditions $k_{1}\geq k_{2}\geq \cdots \geq k_{n}$.
We shall say that $\kk$ is a one-step representation of $K$ if it is of the following form
\begin{equation*}
\kk=(\underbrace{m+\ell,\dots,m+\ell}_{k},\underbrace{m,\dots,m}_{n-k})
\end{equation*}
for $1\le k\le n-1$.

Let $\kk$ be an irreducible finite-dimensional representation of
$\U(n)$. Then $\kk$ is a subrepresentation of $\mm$ if and only if
the coefficients $k_{i}$ satisfy the interlacing property
$$m_{i}\geq k_{i}\geq m_{i+1},\quad  \text { for all
}\quad  i=1, \dots , n.$$ Moreover if $\kk$ is a subrepresentation
of $\mm$ it appears only once. (See \cite{VK}). Therefore the height of any
irreducible spherical function $\Phi$ of $(G,K)$ is one. This is equivalent to the commutativity of the algebra $D(G)^K$.
(See \cite{GV}, \cite{T1}).

The representation space $V_\kk$ of $\kk$ is a subspace of the
representation space $V_\mm$ of $\mm$ and it is also $K$-stable. In
fact, if $A\in\U(n)$, $a=(\det A)^{-1}$ and  $v\in V_\kk$ we have
$$\left(\begin{matrix} A&0\\0&a\end{matrix}\right)\cdot v=
a\left(\begin{matrix} a^{-1}A&0\\0&1\end{matrix}\right)\cdot
v=a^{s_\mm-s_\kk}\left(\begin{matrix}A&0\\0&1\end{matrix}\right)\cdot
v,$$ where $s_\mm= m_{1}+\cdots +m_{n+1}$ and $s_\kk=k_{1}+\cdots
+k_{n}$. This means that the representation of $K$ on $V_\kk$
obtained from $\mm$ by restriction is parameterized by
\begin{equation}\label{Ktipos}
(k_{1}+s_\kk-s_\mm, \dots ,k_{n} + s_\kk-
s_\mm).
\end{equation}

\medskip
Let $\Phi^{\mm, \kk}$ be  the spherical  function of $(G,K)$ associated to the
representation $\mm$ of $G$ and to the subrepresentation $\kk$ of
$\U(n)$. Then \eqref{Ktipos} says that the $K$-type of $\Phi^{\mm, \kk}$ is
$\kk+(s_\kk-s_\mm)(1,\dots,1)$.

\begin{prop}\label{equivalencia}
  The spherical functions $\Phi^{\mm,\kk}$ and
  $\Phi^{\mm',\kk'}$ of the pair $(G,K)$ are equivalent if and only
  if $\mm'=\mm+j(1,\dots, 1)$ and $\kk'=\kk+j(1,\dots, 1)$.
\end{prop}
\begin{proof}
The spherical functions $\Phi^{\mm,\kk}$ and
  $\Phi^{\mm',\kk'}$ are equivalent if and only if
  $\mm$ and $\mm'$ are equivalent and the $K$-types of both spherical functions are the
  same, see the discussion in p. 85 of \cite{T1}.  We know that $\mm\simeq \mm'$ if and only if
$$\mm'=\mm+j(1,\dots, 1)\quad \text{ for some } j\in \ZZ.$$
Besides, the $K$ types are the same if and only if
$$k_{i}+s_\kk-s_\mm= k_{i}'+s_{\kk'}-s_{\mm'}\qquad \text{ for all } i=1,\dots,n.$$
Therefore $\kk'=\kk+p(1,\dots,1)$,  and now it is easy to see that
$p=j$. \qed
\end{proof}

\begin{remark}\label{Ktipo}
Given a spherical function $\Phi^{\mm,\kk}$ we can assume that
$s_\kk-s_\mm=0$. In such a case the $K$-type of $\Phi^{\mm,\kk}$ is $\kk$, see \eqref{Ktipos}.
\end{remark}

Our Lie group  $G$ has the following polar decomposition $G=KAK$, where the abelian subgroup $A$ of $G$ consists of all matrices of the form
\begin{equation}\label{atheta}
a(\theta)= \left(\begin{matrix} \cos \theta& 0& \sin \theta\\ 0&I_{n-1}&0\\ -\sin \theta & 0&\cos
\theta\end{matrix}\right),\qquad \theta\in \RR.
\end{equation}
(Here $I_{n-1}$ denotes the identity matrix of size $n-1$). Since an irreducible spherical function $\Phi$  of $G$ of type $\delta$ satisfies $\Phi(kgk')= \Phi(k)\Phi(g)\Phi(k')$ for all $k,k'\in K$ and $g\in G$, and  $\Phi(k)$ is an irreducible representation of $K$ in the class $\delta$, it follows that $\Phi$ is determined by its restriction to $A$ and its $K$-type. Hence, from now on, we shall consider its restriction to $A$.

Let $M$ be the centralizer of  $A$ in $K$. Then $M$ consists of all
elements of the form
\begin{equation*}\label{M}
m= \left(\begin{matrix} e^{ir}& 0&0\\ 0&B&0\\ 0&
0&e^{ir}\end{matrix}\right), \qquad  r\in \RR, B\in \U(n-1), \det
B=e^{-2ir}.
\end{equation*}

The finite-dimensional irreducible representations of $\U(n-1)$ are parameterized by the $(n-1)$-tuples of integers
$$\mathbf t=(t_{1}, t_{2},\dots, t_{n-1})$$
such that $t_{1}\geq t_{2}\geq \cdots \geq t_{n-1}$.

The representation of $\U(n)$ in $V_\kk\subset V_\mm$,  $\kk=(k_{1}, \dots ,k_{n})$ restricted to $\U(n-1)$ decomposes as the following direct sum
\begin{equation}\label{U(n-1)subrepresentations}
V_\kk=\bigoplus_{\mathbf t\in \hat \U(n-1)} V_{\mathbf t},
\end{equation}
where the sum is over all the representations $\mathbf t=(t_{1}, \dots , t_{n-1})\in \hat \U(n-1)$ such that the coefficients of $\mathbf t$ interlace the coefficients of $\kk$: $k_{i}\geq
t_{i}\geq k_{i+1}$, for all $i=1, \dots , n-1$.

If $a\in A$, then  $\Phi^{\mm,\kk}(a)$ commutes with $\Phi^{\mm,\kk}(m)$ for all $m\in M$. In fact we have
$$\Phi^{\mm,\kk}(a)\Phi^{\mm,\kk}(m)=\Phi^{\mm,\kk}(am)=\Phi^{\mm,\kk}(ma)=\Phi^{\mm,\kk}(m)\Phi^{\mm,\kk}(a).$$

But $m=e^{ir}\left(\begin{matrix} 1& 0&0\\ 0&e^{-ir}B&0\\ 0&
0&1\end{matrix}\right)$ and $e^{ir}I$ is in the center of $\U(n+1)$.
Hence  $V_\mathbf t$ is an irreducible $M$-module and
$\Phi^{\mm,\kk}(a)$ also commutes with the action $\U(n-1)$. Since
 $V_\mathbf t\subset V_\kk$ appears only once, by Schur's Lemma
it follows that $\Phi^{\mm,\kk}(a)|_{V_\mathbf
t}=\phi^{\mm,\kk}_\mathbf t(a)I|_{V_\mathbf t}$, where
$\phi^{\mm,\kk}_ \mathbf t(a)\in\CC$ for all $a\in A$.

\smallskip
For $g\in G$, let $A(g)$ denote the $n\times n$ left upper
corner of  $g$, and let
$$\A=\{g\in G: A(g) \text{ is nonsingular}  \}.$$
Notice that $\mathcal A$ is an  open dense subset of $G$ which is
left and right invariant under $K$. The set $\mathcal A$ can also be described as the set of all $g\in G$ such that $g_{n+1,n+1}\neq 0$. This is a consequence of the following lemma.

\begin{lem}\label{menores} If $U=(u_{ij})\in \SU(n+1)$, we shall denote by $U_{(i\vert
j)}$ the $n\times n$ matrix obtained from $U$ by eliminating the
$i$th row and the $j$th column.  Then
$$\det U_{(i|j)}=(-1)^{i+j}\overline u_{ij}.$$
\end{lem}
\begin{proof}
The adjoint of a square matrix $U$ is the matrix whose $ij$-element
is defined by $(\adj\,U)_{ij}=(-1)^{i+j}\det U_{(j\vert i)}$, and it
is denoted by $\adj\,U$. Then $U\,\adj\,U=\det U$. But if
$U\in\SU(n+1)$ we have $U^{-1}=U^*$ and $\det U=1$, hence
$$(-1)^{i+j}\det U_{(j\vert i)}=\overline u_{ji},$$
which completes the proof of the lemma. \qed
\end{proof}

\medskip
As in \cite{GPT1} to determine
all irreducible spherical functions of $G$ of type $\kk$ an
auxiliary  function $\Phi_\kk:\A\longrightarrow \End(V_\kk)$ is
introduced. It is defined by $$\Phi_\kk(g)=\pi(A(g)),$$ where  $\pi$
stands for the unique holomorphic representation of $\GL(n,\CC)$
corresponding to the parameter $\kk$.  It turns out that if $k_n\geq
0$, then $\Phi_\kk= \Phi^{\mm,\kk}$ where $\mm=(k_1,\dots,k_n,0)$.

Then instead of looking at a general spherical function
$\Phi$ of type $\kk$, we shall look at the
function $H(g)=\Phi(g)\Phi_\kk (g)^{-1}$ which is
well defined on $\A$.

\section{The differential operators $D$ and $E$}

\noindent The group $G$ acts in a natural way on the complex
projective space $P_n(\CC)$. This action is transitive and
$K$ is the isotropy subgroup of the point
$(0,\dots,0,1)\in P_n(\CC)$. Therefore
$$ P_n(\CC)\simeq G/K.$$
Moreover the $G$-action on $P_n(\CC)$
corresponds to the action induced by left multiplication on $G/K$.
We identify the complex space $\CC^n$ with the affine space
$\CC^n=\vzm{(z_1,\dots,z_n,1)\in P_n(\CC)}{(z_1,\dots,z_n)\in \CC^n},$
and we will take full advantage of the $K$-orbit structure of $P_n(\CC)$.
The affine space $\CC^n$ is $K$-stable and the corresponding space
at infinity $L=P_{n-1}(\CC)$ is a $K$-orbit. Moreover the $K$-orbits in $\CC^n$
are the spheres
$$S_r=\vzm{(z_1,\dots,z_n)\in \CC^n}{|z_1|^2+\cdots+|z_n|^2=r^2}.$$ Thus we can take the
points $(r,0,\dots,0)\in S_r$ and $(1,0,\dots,0)\in L$ as
representatives of $S_r$ and $L$, respectively. Since
$(M,0,\dots,0,1)=(1,0,\dots,\textstyle\frac1M)\longrightarrow
(1,0,\dots,0)$ when $M\rightarrow \infty$, the closed interval
$[0,\infty]$ parameterizes the set of $K$-orbits in $P_n(\CC)$.

Let us consider on $\CC^n$ the $2n$-real linear coordinates
$(x_1,y_1,\dots,x_n,y_n)$ defined by:
$x_j(z_1,\dots,z_n)+i\,y_j(z_1,\dots,z_n)=z_j$ for all
$(z_1,\dots,z_n)\in \CC^n$. We also introduce the following usual
notation:
\begin{equation*} \frac{\partial }{\partial z_j}=\frac 12 \left(
\frac{\partial }{\partial x_j}- i\,\frac{\partial }{\partial
y_j}\right).
\end{equation*}

From now on any $X\in\liesl(n+1,\CC)$ will be considered as a left-invariant complex vector field on $G$.
We will be interested in the following
left-invariant differential operators on $G$,
\begin{equation}\label{DeltaPQ}
\Delta_P=\sum_{1\le j\le
n}E_{n+1,j}E_{j,n+1},\quad\Delta_Q=\sum_{1\le i,j\le
n}E_{n+1,i}E_{j,n+1}E_{ij}.
\end{equation}

\begin{lem}\label{Kinvariantes}
The differential operators $\Delta_P$ and $\Delta_Q$ are in $D(G)^K$.
\end{lem}
\begin{proof}
It is clear from the definitions that $\Delta_P$ and $\Delta_Q$ are elements of weight zero with respect to the Cartan subalgebra of $\liek_\CC$ of all diagonal matrices. Thus to prove that $\Delta_P$ and $\Delta_Q$ are right invariant under $K$ it suffices to prove, respectively, that $[E_{r,r+1},\Delta_P]=0$ and that $[E_{r,r+1},\Delta_Q]=0$ for $1\le r\le n-1$. We have
\begin{align*}
[E_{r,r+1},\Delta_P]=&\sum_{j=1}^n[E_{r,r+1},E_{n+1,j}]E_{j,n+1}+\sum_{j=1}^nE_{n+1,j}[E_{r,r+1},E_{j,n+1}]\displaybreak[0]\\
=&-E_{n+1,r+1}E_{r,n+1}+E_{n+1,r+1}E_{r,n+1}=0.
\end{align*}

Similarly,
\begin{align*}
[E_{r,r+1},\Delta_Q]=&\sum_{i,j=1}^n[E_{r,r+1},E_{n+1,i}]E_{j,n+1}E_{ij}+\sum_{i,j=1}^nE_{n+1,i}[E_{r,r+1},E_{j,n+1}]E_{ij}\\
&+\sum_{i,j=1}^nE_{n+1,i}E_{j,n+1}[E_{r,r+1},E_{ij}] \displaybreak[0]\\
=&-\sum_{j=1}^nE_{n+1,r+1}E_{j,n+1}E_{rj}+\sum_{i=1}^nE_{n+1,i}E_{r,n+1}E_{i,r+1}\\
&+\sum_{j=1}^nE_{n+1,r+1}E_{j,n+1}E_{rj}-\sum_{i=1}^nE_{n+1,i}E_{r,n+1}E_{i,r+1}=0.
\end{align*}
\qed
\end{proof}

\subsection{Reduction to $P_n(\CC)$}\label{reduccionproyectivo} \mbox{} \\
According to Proposition \ref{defeq} if $\Phi=\Phi^{\mm,\kk}$ denotes a generic irreducible spherical function of type $\kk$, then $\Delta_P\Phi=\lambda\Phi$ and $\Delta_Q\Phi=\mu\Phi$ with $\lambda,\mu\in\CC$, because $[\Delta_P\Phi](e)=\lambda I$ and  $[\Delta_Q\Phi](e)=\mu I$ since $V_\mm$ as a $K$-module is multiplicity-free.

We introduced the auxiliary  function $\Phi_\kk:\A\longrightarrow \End(V_\kk)$. It is defined by $\Phi_\kk(g)=\pi(A(g))$ where  $\pi$ stands for the unique holomorphic representation of $\GL(n,\CC)$ of highest weight $\kk$.
Then instead of looking at a general spherical function
$\Phi$ of type $\kk$ we look at the
function
$$H(g)=\Phi(g)\Phi_\kk (g)^{-1}$$
which is well defined on $\A$.
Then $H$ satisfies
\begin{enumerate}
\item [(i)] $H(e)=I$.
\item [(ii)]$ H(gk)=H(g)$, for all $g\in {\mathcal A}, k\in K$.
\item [(iii)]
$H(kg)=\pi(k)H(g)\pi(k^{-1})$, for all
$g\in {\mathcal A}, k\in K$.
\end{enumerate}

\smallskip
The projection map $p:G\longrightarrow P_n(\CC)$ defined by $p(g)=g\cdot(0,\dots,0,1)$, maps the
open set ${\mathcal A}$ onto the affine space $\CC^n$. Thus (ii) says that $H$
may be considered as a function on $\CC^n$.

The fact that $\Phi$ is an eigenfunction of $\Delta_P$ and
$\Delta_Q$ makes $H$ into an eigenfunction of certain differential
operators $D$ and $E$ on $\CC^n$, to be determined now.

Since the function $A$  is the restriction of a holomorphic function defined on an open subset of $\GL(n+1,\CC)$, and $A(g\exp tE_{j,n+1})=A(g)$ for all $t\in\RR$ and $1\le j\le n$, it follows that $E_{j,n+1}(\Phi_\kk)=0$. Thus
\begin{equation*}\label{D2}
\Delta_P(H\Phi_\kk)=\sum_{j=1}^n\big((E_{n+1,j}E_{j,n+1}H)\Phi_\kk+(E_{j,n+1}H)(E_{n+1,j}\Phi_\kk)\big).
\end{equation*}

Let $\dot\pi$ be the irreducible representation of $\liegl(n,\CC)$ obtained by derivation of $\pi$ of highest weight $\kk$.
Since $\Phi_\kk(gk)=\Phi_\kk(g)\pi(k)$ for all $g\in\mathcal A, k\in K$ we have $E_{i,j}\Phi_\kk=\Phi_\kk\dot\pi(E_{i,j})$.  Therefore
\begin{equation*}\label{E2}
\Delta_Q(H\Phi_\kk)=\sum_{i,j=1}^n\big((E_{n+1,i}E_{j,n+1}H)\Phi_\kk
+(E_{j,n+1}H)(E_{n+1,i}\Phi_\kk)\big)\dot\pi(E_{ij}).
\end{equation*}

In the open set ${\mathcal A} \subset G$ let us consider the following
differential operators. For
 $H\in C^{\infty}({\mathcal A})\otimes \End(V_\kk) $ let

\begin{equation*}\label{DeltaP1P2}
\begin{split}
D_1 H &=\sum_{j=1}^n E_{n+1,j}E_{j,n+1}H,\;\\
D_2 H &=\sum_{j=1}^n(E_{j,n+1}H)(E_{n+1,j}\Phi_\kk)\Phi_\kk^{-1},
\end{split}
\end{equation*}
\begin{equation*}\label{DeltaQ1Q2}
\begin{split}
E_1 H& =\sum_{i,j=1}^n(E_{n+1,i}E_{j,n+1}H)\Phi_\kk\dot\pi(E_{ij})\Phi_\kk^{-1},
\\
E_2 H &=\sum_{i,j=1}^n(E_{j,n+1}H)(E_{n+1,i}\Phi_\kk)\dot\pi(E_{ij})\Phi_\kk^{-1}.
\end{split}
\end{equation*}

We observe that
\begin{equation}\label{DeltaHPhi}
  \Delta_P(H\Phi_\kk) = \Big ( D_1(H)+D_2(H)\Big) \Phi_\kk \text{ and } \Delta_Q(H\Phi_\kk) = \Big ( E_1(H)+E_2(H)\Big) \Phi_\kk;
\end{equation}
then it is clear that $\Phi=H\Phi_\kk$ is an eigenfunction of $\Delta_P$ and $\Delta_ Q$ if and only if $H$ is an eigenfunction of  $D=D_1+D_2$ and $E=E_1+E_2$.

\begin{lem} The differential operators $D_1$, $D_2$, $E_1$, and $E_2$ define differential operators
$D_1, D_2, E_1$ and $E_2$ acting on $C^\infty(\CC^n)\otimes\End(V_\kk)$.
\end{lem}
\begin{proof}
The only thing we really need to prove is that $D_1$, $D_2$, $E_1$, and $E_2$ preserve the subspace  $C^{\infty}({\mathcal A})^K\otimes\End(V_\kk)$ of all right $K$-invariant functions.

For $H\in C^{\infty}({\mathcal A})\otimes\End(V_\kk)$ and $k\in K$  we write $H^{R(k)}(g)=H(gk)$ for all $g\in \mathcal A$. Then $(Ad(k)D)H=(DH^{R(k^{-1})})^{R(k)}$  for any $D\in D(G)$.   In particular any $D\in D(G)^K$ leaves  $C^{\infty}({\mathcal A})^K\otimes\End(V_\kk)$ invariant.

That $D_1$ has this property is a consequence of $D_1H=\Delta_PH$ and $\Delta_P\in D(G)^K$.
On the other hand, from \eqref{DeltaHPhi} we get
$$\Delta_P(H\Phi_\kk)\Phi_\kk^{-1}=D_1H+D_2 H.$$
If $H\in C^\infty(\A)^K\otimes\End(V_\kk)$, it is easy to verify that  $\Delta_P(H\Phi_\pi)\Phi_\kk^{-1}$ belongs to $C^{\infty}({\mathcal A})^K\otimes\End(V_\kk)$. Therefore $D_2$ also preserves $C^{\infty}({\mathcal A})^K\otimes\End(V_\kk)$.

In a similar way from \eqref{DeltaHPhi} we obtain
$$\Delta_Q(H\Phi_\kk))\Phi_\kk^{-1}=E_1H+E_2H.$$
Since $\Delta_Q\in D(G)^K$ it follows that $\Delta_Q(H\Phi_\pi)\Phi_\kk^{-1}\in C^{\infty}({\mathcal A})^K\otimes\End(V_\kk)$ for all $H\in C^\infty(\A)^K\otimes\End(V_\kk)$. Hence to finish the proof of the lemma it suffices to prove that $E_1$ leaves invariant  $C^{\infty}({\mathcal A})^K\otimes\End(V_\kk)$.

A computation similar to the one done in the proof of Lemma \ref{Kinvariantes} shows that the element
$$\sum_{i,j=1}^n E_{n+1,i}E_{j,n+1}\otimes E_{ij}\,\in D(G)\otimes D(K)$$
is $K$-invariant. Now let us consider the unique linear map
$$L:D(G)\otimes D(K)\longrightarrow \End\big(C^\infty(\A)^K\otimes\End(V_\kk)\big)$$
such that $L(E\otimes F)(H)=(EH)\Phi_\kk\dot\pi(F)\Phi_\kk^{-1}$. Then
$$\Delta_{Q_1}H=L\Big(\sum_{i,j=1}^n E_{n+1,i}E_{j,n+1}\otimes E_{ij}\Big)(H).$$
Thus if $k\in K$ and $H\in C^\infty(\A)^K\otimes \End(V_\kk)$, then
\begin{align*}
\Delta_{Q_1}H&=L\Big(\sum_{i,j=1}^n Ad(k)(E_{n+1,i}E_{j,n+1})\otimes Ad(k) E_{ij}\Big)(H)\\
&=\sum_{i,j=1}^n(Ad(k)( E_{n+1,i}E_{j,n+1})H)\Phi_\kk\pi(k)\dot\pi(E_{ij})\pi(k^{-1})\Phi_\kk^{-1}\\
&=\sum_{i,j=1}^n( E_{n+1,i}E_{j,n+1})H)^{R(k)}\Phi_\kk^{R(k)}\dot\pi(E_{ij})(\Phi_\kk^{-1})^{R(k)}=(\Delta_{Q_1}H)^{R(k)}.
\end{align*}
This completes the proof of the lemma.
\qed
\end{proof}

The proofs of Propositions \ref{expD1}, and \ref{expE1}
 require simple but lengthy computations. Complete details will be given in \cite{PT5}.

Let $\dot\pi$ be  the irreducible representation of $\liegl(n,\CC)$
of highest weight $\kk$.

\begin{prop}\label{expD1}
For $H\in C^\infty(\CC^n)\otimes \End(V_\kk)$ we have
\begin{align*}
DH& =-\frac14\Bigl(1+\sum_{1\leq j\leq n} |z_j|^2\Bigr)\biggl(\,
\sum_{1\leq r\leq n}
(H_{x_rx_r}+H_{y_ry_r})(1+|z_r|^2)\\
& \quad+\sum_{1\leq r\neq k\leq n} (H_{x_rx_k}+H_{y_ry_k}) \re
(z_r\overline z_k)
-2\sum_{1\leq r\neq k\leq n} H_{x_ry_k} \im (z_r\overline z_k)\biggr)\\
& \quad + \sum_{1\leq r\leq n}\,\frac{\partial H}{\partial z_r} \,\dot\pi\Big ( \sum_{1\leq i,j\leq n} z_i(\delta_{rj}+z_r\overline z_j) E_{ij}\Big),
\end{align*}
\end{prop}

\smallskip
\begin{prop}\label{expE1} For $H\in C^\infty(\CC^n)\otimes \End(V_\kk)$ we have
\begin{align*}
EH&=-\frac14\Bigl(1+\sum_{1\leq j\leq n} |z_j|^2\Bigr)\biggl(
\sum_{1\leq r,k\leq n}(H_{x_rx_k}+H_{y_ry_k}) \dot\pi\Bigl(\sum_{1\leq j\leq n}(\delta_{kj}+z_k\overline z_j)E_{rj}
\Bigr)\\
&+i\sum_{1\leq r,k\leq n} H_{x_ry_k} \dot\pi\Bigl(\sum_{1\leq j\leq n}(\delta_{jr}+z_r\overline z_j)E_{kj}
-(\delta_{kj}+z_k\overline z_j)E_{rj}\Bigr)\biggr)\\ &
+ \sum_{1\leq r,k\leq n}\frac{\partial H}{\partial z_r}\dot\pi\Bigl(\sum_{1\leq j\leq n} z_jE_{jk}\Bigr)
\dot\pi\Bigl(\sum_{1\leq j\leq n}(\delta_{rj}+z_r\overline z_j)E_{kj}\Bigr).
\end{align*}
\end{prop}

\subsection{Reduction to one variable }\label{secc1var} \mbox{} \\
 We are
interested in considering the differential operators $D$ and $E$ applied to a
function $H\in C^\infty(\CC^n)\otimes \End(V_\pi)$ such that
$H(kp)=\pi(k)H(p)\pi(k)^{-1}$, for all $k\in K$ and $p$ in the
affine complex space $\CC^n$. This property of $H$ allows us to find
ordinary differential operators $\tilde D$ and $\tilde E$ defined on the interval
$(0,\infty)$ such that
$$(D\,H)(r,0,\dots ,0)=(\tilde D\tilde H)(r),\quad (E\,H)(r,0,\dots ,0)=(\tilde E\tilde H)(r),$$
where $\tilde H(r)=H(r,0,\dots, 0)$. Let $(r,\mathbf 0)=(r,0,\dots, 0)\in \CC^n$.
We also introduce
differential operators $\tilde D_1$, $\tilde D_2$, $\tilde E_1$ and $\tilde E_2$  in the same way, that is  $(D_1 H)(r,\mathbf 0)=(\tilde D_1\tilde H)(r)$, $(D_2
H)(r,\mathbf 0)=(\tilde D_2\tilde H)(r)$, $(E_1 H)(r,\mathbf 0)=(\tilde E_1\tilde H)(r)$   and $(E_2 H)(r,\mathbf 0)=(\tilde E_2\tilde H)(r)$.

In order to give the explicit expressions of $\tilde D$ and $\tilde E$ we need to
compute a number of second order partial derivatives at the points $(r,\mathbf 0)$ of the function
$H:\CC^n\longrightarrow \End(V_\pi)$. We recall that $[0,\infty]$ parameterizes the set of $K$-orbits in $P_n(\CC)$.
Given $z_1, z_j\in\CC$,  $ 2\leq j\leq n$, we need a matrix in $K$ such that carries the point $(z_1,0,\dots,0,z_j,0,\dots,0)\neq 0$ to the meridian
$\{(r, \mathbf 0) : r>0\}$.
A good choice is the following $n\times n$ matrix
\begin{equation}\label{A2x2}
A(z_1,z_j)=\frac1{s(z_1,z_j)}\Big(z_1E_{11}-\overline z_jE_{1j}+z_jE_{j1}+\overline z_1E_{jj}+\sum_{k\ne 1,j}E_{kk}\Big)\in \SU(n),
\end{equation}
where  $s(z_1,z_j)=\sqrt{|z_1|^2+|z_j|^2}\ne0$.
Then
$$(z_1,0,\dots,0,z_j,0,\dots,0)^t=A(z_1,z_j)(s(z_1,z_j),0,\dots,0)^t.$$

The proofs of the following theorems  are similar as in the case of the complex projective plane, considered in \cite{GPT1}.
Complete details will appear in the forthcoming paper \cite {PT5}.

\begin{thm}\label{main'} For all $r>0$ we have
\begin{align*}
\tilde D\tilde H&= -\frac 14 \Bigg((1+r^2)^2 \frac{d^2\tilde H}{dr^2}(r)+
\frac{(1+r^2)}r \frac{d\tilde H}{dr}(r)\left( 2n-1+r^2-2r^2\dot\pi(E_{11}) \right)\\
& +\frac{4}{r^2}\sum_{2\leq j \leq n} \big [\dot\pi(E_{j1}), \tilde H(r) \big] \dot\pi(E_{1j})+\frac{4(1+r^2)}{r^2} \sum_{2\leq j \leq
n}\big[\dot\pi(E_{1j}),\tilde H(r)\big ]\dot\pi(E_{j1})\Bigg).
\end{align*}
\end{thm}

\smallskip
\begin{thm}\label{main''} For all $r>0$ we have

\begin{align*}
\tilde E \tilde H =& -\frac{1}4\Bigg((1+r^2)^2\frac{d^2\tilde H}{dr^2}\dot\pi(E_{11}) +\frac{(1+r^2)}r \big(2n-1+r^2-2r^2\dot\pi(E_{11})\big) \frac{d\tilde H}{dr}\dot\pi(E_{11})
\\
&
+\frac{2(1+r^2)}{r}\sum_{ j=2}^ n \dot\pi(E_{j1})\frac{d\tilde H}{dr}\dot\pi(E_{1j})
-\frac{2(1+r^2)^2}{r}\sum_{j=2}^n \dot\pi(E_{1j})\frac{d\tilde H}{dr}\dot\pi(E_{j1})
\\
&-\frac{2(1+r^2)}{r^2}\sum_{j,k=2}^n\Big(\big[\dot\pi(E_{k1}),\big[\dot\pi(E_{1j}),\tilde
H\big]\big]+\big[\dot\pi(E_{1j}),\big[\dot\pi(E_{k1}),\tilde
H\big]\big]\Big)\dot\pi(E_{jk}) \\
&
-4\sum_{k=1}^n\sum_{j=2}^n[\dot\pi(E_{j1}),\tilde H]\dot\pi(E_{1k})\dot\pi(E_{kj}) \Bigg).
\end{align*}

\end{thm}

\subsection{The operators $\tilde D$ and $\tilde E$ in matrix form}
\mbox{} \\
From now on we will assume that the $K$-types of our spherical functions will be of
the one step kind, i.e., for $1\le k\le n-1$ let
\begin{equation}\label{onestep1}
\kk=(\underbrace{m+\ell,\dots,m+\ell}_{k},\underbrace{m,\dots,m}_{n-k}).
\end{equation}

Therefore, the irreducible spherical functions $\Phi$ of the pair $(G,K)$, whose
$K$-type is as in \eqref{onestep1}, are parameterized by the $n+1$-tuples  $\mm$ of the form

\begin{equation}\label{n+1tupla}
\mm=\mm(w,r)=(w+m+\ell,\underbrace{m+\ell,\dots,m+\ell}_{k-1},m+r,\underbrace{m,\dots,m}_{n-k-1},-w-r ),
\end{equation}
where  $0\leq w$,  $m\geq -w-r$ and $0\leq r \leq\ell$. See Proposition \ref{equivalencia} and Remark \ref{Ktipo}.
Thus if we assume $w\ge w_0=\max\{0,-m\}$ and  $0\leq r \leq\ell$ all the conditions are satisfied.

We observe now that the decomposition \eqref{U(n-1)subrepresentations}
of $V_\kk$ into $\U(n-1)$-submodules is
\begin{equation}\label{onestep2}
V_\kk=\bigoplus_{s=0}^\ell V_{\mathbf t(s)},
\end{equation}
where
\begin{equation*}
\mathbf
t(s)=(\underbrace{m+\ell,\dots,m+\ell}_{k-1},m+s,\underbrace{m,\dots,m}_{n-k-1}).
\end{equation*}

\smallskip
We recall that for any $0<r<\infty$, $\tilde H(r)$ is a linear
operator of $V_\kk$ which commutes with the action of $U(n-1)$. Thus
$\tilde H(r)$ is a scalar transformation on each $U(n-1)$-submodule
$V_{\mathbf t(s)}$. We denote this scalar by  $\tilde h_s(r)$ and we
will identify the function $\tilde H(r)$ associated to $\Phi$ with
the column vector function:
$$ \tilde H(r)=(\tilde h_0(r), \dots , \tilde h_\ell(r))^t.$$

The expression of the differential operator $\tilde D$ in Theorem \ref{main'} is given
in terms of linear operators of the vector space $V_\kk$. Now we proceed to give $\tilde D$
as a system of linear differential operators in  the $\ell+1$ unknowns $\tilde h_0(r),\dots,\tilde h_\ell(r)$.

\begin{prop}\label{matrixD} For all $r>0$ and $0\le s\le\ell$ we have
\begin{align*}
  ( \tilde D \tilde H)_s(r)= & -\frac{(1+r^2)^2}4\tilde h_s''(r)-\frac{(1+r^2)}{4r}
\bigl(2n-1+r^2-2(m+\ell-s)r^2\bigr)\tilde h_s'(r)
  \\ &-
   \frac{1}{r^2}\,s(\ell+k-s) \bigl(\tilde h_{s-1}(r)-\tilde h_{s}(r)\bigr)\\
   &
      -\frac{(1+r^2)}{r^2}(\ell-s)(n-k+s) \bigl(\tilde h_{s+1}(r)-\tilde h_s(r)\bigr).
\end{align*}
\end{prop}
\begin{proof}
To obtain the proposition from Theorem \ref{main'}  we need to
compute in each $V_{\mathbf t(s)}$, $0\le s\le\ell$, the linear
transformations
\begin{align*}
&\sum_{2\leq j\leq n}\big[ \dot\pi(E_{j1}),\tilde H(r)\big ] \dot\pi(E_{1j}), \qquad
\sum_{2\leq j\leq n} \big[ \dot\pi(E_{1j}), \tilde H(r) \big ]\dot\pi(E_{j1}).
\end{align*}

The highest weight of $V_{\mathbf t(s)}$ is
\begin{equation}\label{mus}
\begin{split}
\mu_s =&\;(m+\ell-s)x_1+(m+\ell)(x_2+\cdots +x_{k})+(m+s)x_{k+1}\\
&+m(x_{k+2}+\cdots + x_n).
\end{split}
\end{equation}
 Therefore all
weights of $V_{\mathbf t(s)}$ are of the form
$\mu=\mu_s-\sum_{r=2}^{n-1} n_r\alpha_r$ with
$\alpha_r=x_r-x_{r+1}$, thus $\mu= (m+\ell-s)x_1+ \cdots$. Now we
observe that for the representations $\pi$ considered here, for all
$j=2,\dots , n$, we have
\begin{equation}\label{E1j}
\dot\pi(E_{1j})(V_{\mathbf t(s)})\subset V_{\mathbf t(s-1)} \quad \text{and} \quad
\dot\pi(E_{j1})(V_{\mathbf t(s)})\subset V_{\mathbf t(s+1)}.
\end{equation}
In fact, if $v\in V_{\mathbf t(s)}$ is a vector of weight $\mu$, then
$\dot\pi(E_{1j})v$ is a vector of weight
$\mu+x_1-x_j=(m+\ell-(s-1))x_1+ \cdots$. Similarly
$\dot\pi(E_{j1})v$ is a vector of weight
$\mu+x_j-x_1=(m+\ell-(s+1))x_1+ \cdots$.

\noindent Therefore
\begin{equation*}
\begin{split}
\sum_{2\leq j\leq n}
\big [\dot\pi(E_{j1}),\tilde H(r)\big] \dot\pi(E_{1j})v&= \sum_{2\leq j\leq n}
\big( \dot\pi(E_{j1})\tilde H(r)- \tilde H(r) \dot\pi(E_{j1})\big) \dot\pi(E_{1j})\\
& =\big(\tilde h_{s-1}(r)-\tilde h_s(r)\big)\sum_{2\leq
j\leq n} \dot\pi(E_{j1})\dot\pi(E_{1j})v
\end{split}
\end{equation*}
 and
\begin{equation*}
\sum_{2\leq j\leq n}
\dot\pi(E_{1j})\tilde H(r)\dot\pi(E_{j1})v=\big( \tilde h_{s+1}(r)- \tilde h_s(r)\big) \sum_{2\leq
j\leq n} \dot\pi(E_{1j})\dot\pi(E_{j1})v.
\end{equation*}

The Casimir element of
${\mathrm GL}(n,\CC)$ is
$$\Delta^{(n)}= \sum_{1\leq i,j \leq n} E_{ij}E_{ji}= \sum_{1\leq i\leq
n} E_{ii}^2+ \sum_{1\leq i<j\leq n} (E_{ii}-E_{jj})+ 2\sum_{1\leq
i<j\leq n} E_{ji}E_{ij}.$$
Similarly, the Casimir operator of
${\mathrm GL}(n-1,\CC)\subset {\mathrm GL}(n,\CC)$ is
$$\Delta^{(n-1)}=\sum_{2\leq i,j \leq n} E_{ij}E_{ji}=  \sum_{2\leq i\leq
n} E_{ii}^2+ \sum_{2\leq i<j\leq n} (E_{ii}-E_{jj})+ 2\sum_{2\leq
i<j\leq n} E_{ji}E_{ij}.$$

\noindent Hence
\begin{equation}\label{operador}
\sum_{2\leq j\leq n} E_{j1}E_{1j}= \tfrac 12
\Bigl(\Delta^{(n)}-\Delta^{(n-1)} -E_{11}^2- \sum_{2\leq j\leq n}(E_{11}-E_{jj}) \Bigr).
\end{equation}
We also have
\begin{equation}\label{operador2}
\sum_{2\leq j\leq n}E_{1j}E_{j1}=\sum_{2\leq j\leq n}E_{j1}E_{1j}+\sum_{2\leq j\leq n}(E_{11}-E_{jj}).
\end{equation}

To compute the scalar linear transformation  $\sum_{2\leq j\leq n}
\dot\pi(E_{j1})\dot\pi(E_{1j})$ on $V_{\mathbf t(s)}$ it is enough
to apply it to a highest weight vector $v_s$ of $V_{\mathbf t(s)}$.
The highest weight of $V_{\kk}$ is $\mu_0$ and the weight of
$v_{s}$ is $\mu_s$, see \eqref{mus}.
 Then we have
\begin{equation}  \label{pipuntoE11}
\begin{split}
  \dot\pi(\Delta^{(n)})v_s=& \bigl(k(m+\ell)^2+(n-k)m^2+k(n-k)\ell \bigr)v_s,  \\
  \dot\pi(\Delta^{(n-1)})v_s=& \bigl((k-1)(m+\ell)^2+ (m+s)^2+(n-k-1)m^2 \\
 & +(\ell-s)(k-1)+s(n-k-1)+\ell(k-1)(n-k-1) \bigr) v_s,\\
\dot\pi(E_{11})v=&\; (m+\ell-s)v_s, \\
\sum_{2\leq j\leq n}\dot\pi(E_{jj}) v_s=& (mn+\ell k-\ell-m+s)v_s.
\end{split}
\end{equation}

\noindent Therefore, by using \eqref{operador} and \eqref{operador2}, we
obtain for all $v\in V_{\mathbf t(s)}$
\begin{equation}\label{piE1jEj1}
\begin{split}
 \sum_{2\leq j\leq n} \dot\pi(E_{j1})\dot\pi(E_{1j})v& =
\dot\pi\Big( \sum_{2\leq j\leq n} E_{j1}E_{1j} \Big )v=s(\ell+k-s)v,\\
 \sum_{2\leq j\leq n} \dot\pi(E_{1j})\dot\pi(E_{j1})v&=(\ell-s)(n-k+s)v.
\end{split}
\end{equation}

Finally we compute, for all $v\in V_{\mathbf t(s)}$
\begin{align}
 \sum_{2\leq j\leq n}
\big [\dot\pi(E_{j1}),\tilde H(r)\big] \dot\pi(E_{1j})v& = s(\ell+k-s) \big(\tilde h_{s-1}(r)-\tilde h_s(r)\big) v, \label{pipuntoEj1} \\
 \sum_{2\leq j\leq n}
\big [\dot\pi(E_{1j}),\tilde H(r)\big] \dot\pi(E_{j1})v& = (\ell-s)(n-k+s) \big(\tilde h_{s+1}(r)-\tilde h_s(r)\big) v, \label{pipuntoE1j}
\end{align}
and now the proposition follows easily from Theorem \ref{main'}.
\qed
\end{proof}

\

To obtain a similar result for the operator $E$ from Theorem \ref{main''},  we need to
compute  the linear transformations
$$\sum_{2\leq j,k\leq n}\dot\pi(E_{k1})\dot\pi(E_{1j})\dot\pi(E_{jk}) \quad \text { and }\quad
\sum_{2\leq j,k\leq n}\dot\pi(E_{1j})\dot\pi(E_{k1})\dot\pi(E_{jk})$$
in each $V_{\mathbf t(s)}$, $0\le s\le\ell$. This is the content of the following lemma.

\begin{lem}\label{F7-F8}
Let us consider the following elements of the universal enveloping algebra of $\liegl(n,\CC)$:
$$F=\sum_{2\leq j,k\leq n}E_{k1}E_{1j}E_{jk}\qquad \text{and} \qquad  F'=\sum_{2\leq j,k\leq
n}E_{1j}E_{k1}E_{jk}.$$
Then they are $\mathrm U(n-1)$-invariant and for $0\le s\le\ell$, we have
\begin{align*}
\dot\pi(F)_{\vert_{V_{\mathbf t(s)}}}& =-s(\ell + k - s)(k - s - m - n+1)I_s,\\
\dot\pi(F')_{\vert_{V_{\mathbf t(s)}}}& =-(\ell-s)(n-k+s)(k-s-m-1)I_s,
\end{align*}
where $I_s$ stands for the identity linear transformation of $V_{\mathbf t(s)}$.
\end{lem}
\begin{proof}
   It is easy to check that the following elements
$$F^{(n)}= \sum_{1\leq i,j,k \leq n} E_{ki}E_{ij}E_{jk} \quad \text{ and } \quad F^{(n-1)}= \sum_{2\leq i,j,k \leq n} E_{ki}E_{ij}E_{jk}$$
are, respectively, in the centers of the universal enveloping algebras of $\liegl(n,\CC)$ and $\liegl(n-1,\CC)$.
We also have
\begin{equation*}\label{F0}
F^{(n)}-F^{(n-1)}=\sum_{1\leq j,k \leq n}
E_{k1}E_{1j}E_{jk}+\sum_{\substack{2\leq i \leq n \\ 1\leq k\leq n}}
E_{ki}E_{i1}E_{1k}+\sum_{2\leq i,j \leq n} E_{1i}E_{ij}E_{j1}.
\end{equation*}

\noindent Moreover
\begin{align*}
\sum_{1\leq j,k \leq n} E_{k1}E_{1j}E_{jk}&=F+\sum_{1\leq k \leq n}
E_{k1}E_{11}E_{1k}+\sum_{2\leq j \leq n} E_{11}E_{1j}E_{j1}\\
=F+E_{11}^3&+(2E_{11}+1)\sum_{2\leq k \leq n} E_{k1}E_{1k}+E_{11}\sum_{2\leq k \leq n}(E_{11}-E_{jj}),\displaybreak[0]\\
\sum_{\substack {2\leq i \leq n\\ 1\leq k\leq n}}
E_{ki}E_{i1}E_{1k}&=\sum_{2\leq i,k\leq n}
E_{ki}E_{i1}E_{1k}+\sum_{2\leq i\leq n} E_{1i}E_{i1}E_{11}\\
&=F+E_{11}\sum_{2\leq i\leq n} (E_{11}-E_{ii})+E_{11}\sum_{2\leq i\leq n} E_{i1}E_{1i},\displaybreak[0]\\
\sum_{2\leq i,j \leq n} E_{1i}E_{ij}E_{j1}&=(n-1)\sum_{2\leq i\leq n} E_{1i}E_{i1}+F'\\
&=(n-1)\sum_{2\leq i\leq n}(E_{11-}E_{ii})+(n-1)\sum_{2\leq i\leq n} E_{i1}E_{1i}+F'.
\end{align*}

\noindent Therefore
\begin{equation}\label{F1}
\begin{split}
F^{(n)}-F^{(n-1)}=&\;2F+F'+(3E_{11}+n)\sum_{2\leq j\leq n}
E_{j1}E_{1j}+E_{11}(E_{11}+n-1)^2\\
&-(2E_{11}+n-1)\sum_{2\leq j\leq n}E_{jj}.
\end{split}
\end{equation}

\noindent We also observe that
\begin{equation}\label{F2}
\begin{split}
F'&=\sum_{2\leq j,k\leq n}E_{1j}E_{k1}E_{jk}
= \sum_{2\leq j,k\leq  n}E_{k1}E_{1j}E_{jk}+ \sum_{2\leq j,k\leq n}(\delta_{jk} E_{11}-E_{kj})E_{jk}\\
&=F- \Delta^{(n-1)}+E_{11}\sum_{2\leq j\leq n}E_{jj}.
\end{split}
\end{equation}

\noindent From \eqref{F1} and \eqref{F2} we obtain
\begin{equation}\label{F3}
\begin{split}
3F=&\;F^{(n)}-F^{(n-1)}+\Delta^{(n-1)}-(3E_{11}+n)\sum_{2\leq j\leq
n}
E_{j1}E_{1j}\\
&-E_{11}(E_{11}+n-1)^2 +(E_{11}+n-1)\sum_{2\leq j\leq n}E_{jj}.
\end{split}
\end{equation}

To compute $\dot\pi(F)$ on the $M$-module $V_{\mathbf t(s)}$ it is
enough to know $\dot\pi(F)v_s$, since $F$ is $M$-invariant. Thus we
only need to determine $\dot\pi(F^{(n)})v_s$ and
$\dot\pi(F^{(n-1)})v_s$ because the other terms have been computed
before in \eqref{pipuntoE11} and \eqref{piE1jEj1} .

Since $F^{(n)}$ is $K$-invariant it is enough to compute
$\dot\pi(F^{(n)})v_0$. But $v_0$ is a highest weight vector of
$\liegl(n,\CC)$ in $V_\kk$; hence it is enough to compute $F^{(n)}$
modulo the left ideal $U(\liegl(n,\CC))\liek_+$ of the universal
enveloping algebra of $\liegl(n,\CC)$ generated by the set
$\{E_{ij}:1\le i<j\le n\}$. Similarly, since $v_s$ is a highest
weight vector of $\liegl(n-1,\CC)$ in $V_{\mathbf t(s)}$ it is enough
to compute $F^{(n-1)}$ modulo the left ideal
$U(\liegl(n-1,\CC))\liem_+$ of $U(\liegl(n-1,\CC))$  generated by
the set $\{E_{ij}:2\le i<j\le n\}$.

We start with $F^{(n)}=\sum_{1\leq i,j,k \leq n}E_{ki}E_{ij}E_{jk}$.
To rewrite $E_{ki}E_{ij}E_{jk}$ we partition the set of indices into
the following subsets, and we use the symbol $\equiv$ to denote
congruence modulo the left ideal $U(\liegl(n,\CC))\liek_+$.
\begin{equation}\label{partition}
\begin{split}
j<k:&\; E_{ki}E_{ij}E_{jk}\equiv0,\\
j=k>i:&\;  E_{ki}E_{ik}E_{kk}=E_{ki}(E_{ik}+E_{kk}E_{ik})\equiv0,\\
j=k=i:&\;  E_{kk}E_{kk}E_{kk}=E_{kk}^3,\\
j=k<i:&\;  E_{ki}E_{ik}E_{kk}=(E_{kk}-E_{ii}+E_{ik}E_{ki})E_{kk}\\
           &\; =(E_{kk}-E_{ii})E_{kk}+E_{ik}(-E_{ki}+E_{kk}E_{ki})\equiv(E_{kk}-E_{ii})E_{kk},\\
i>j>k:&\;
E_{ki}E_{ij}E_{jk}=(E_{kj}+E_{ij}E_{ki})E_{jk}=E_{kk}-E_{jj}+E_{jk}E_{kj}\\
&\; +E_{ij}(-E_{ji}+E_{jk}E_{ki})\equiv E_{kk}-E_{jj},\\
i=j>k:&\;
E_{kj}E_{jj}E_{jk}=(E_{kj}+E_{jj}E_{kj})E_{jk}=E_{kk}-E_{jj}+E_{jk}E_{kj}\\
&\;
+E_{jj}(E_{kk}-E_{jj}+E_{jk}E_{kj})\equiv(E_{kk}-E_{jj})(E_{jj}+1),
\end{split}
\end{equation}
\begin{equation*}
\begin{split}
j>i>k:&\; E_{ki}E_{ij}E_{jk}=E_{ki}(E_{ik}+E_{jk}E_{ij})\equiv
E_{kk}-E_{ii}+E_{ik}E_{ki}\\
&\; \equiv E_{kk}-E_{ii},\\
j>i=k:&\;
E_{kk}E_{kj}E_{jk}=E_{kk}(E_{kk}-E_{jj}+E_{jk}E_{kj})\equiv
E_{kk}(E_{kk}-E_{jj}),\\
j>k>i:&\; E_{ki}E_{ij}E_{jk}=E_{ki}(E_{ik}+E_{jk}E_{ij})\equiv0.
\end{split}
\end{equation*}

Therefore
\begin{equation}\label{F4}
\begin{split}
F^{(n)}\equiv &\;\sum_{1\le k\le n}E_{kk}^3+\sum_{1\le k<j\le
n}2(E_{kk}-E_{jj})E_{kk}\\
& +\sum_{1\le k<j\le n}2(n-j)(E_{kk}-E_{jj})
+\sum_{1\le k<j\le n}(E_{kk}-E_{jj})(E_{jj}+1)\\
=& \sum_{1\le k\le n}E_{kk}^3+\sum_{1\le k<j\le
n}(E_{kk}-E_{jj})(2E_{kk}+E_{jj}+2(n-j)+1).
\end{split}
\end{equation}

In a similar way for $F^{(n-1)}=\sum_{2\leq i,j,k \leq
n}E_{ki}E_{ij}E_{jk}\in U(\liegl(n-1,\CC))$, it is enough to work
modulo the left ideal $U(\liegl(n-1,\CC))\liem_+$. Then we obtain
\begin{equation}\label{F5}
F^{(n-1)}\equiv\sum_{2\le k\le n}E_{kk}^3+\sum_{2\le k<j\le
n}(E_{kk}-E_{jj})(2E_{kk}+E_{jj}+2(n-j)+1).
\end{equation}

The highest weight of $V_\kk$ is $\mu_0
=(m+\ell)(x_1+\cdots+x_{k})+m(x_{k+1}+\cdots + x_n)$. Hence, from
\eqref{F4} it is easy to conclude that
\begin{equation}\label{F6}
\dot\pi(F^{(n)})=\left(k(m+\ell)^3+(n-k)m^3+k\ell(n-k)(2(m+\ell)+m+n-k)\right)I,
\end{equation}
where $I$ stands for the identity linear transformation of $V_\kk$.

The highest weight of $V_{\mathbf t(s)}$ is $$\mu_s
=(m+\ell-s)x_1+(m+\ell)(x_2+\cdots +x_{k})+(m+s)x_{k+1}
+m(x_{k+2}+\cdots + x_n).$$
Hence, from \eqref{F5} it is easy to
conclude that
\begin{equation*}\label{F66}
\begin{split}
\dot\pi(F^{(n-1)})=&\;
\big((k-1)(m+\ell)^3+(m+s)^3+(n-k-1)m^3\\
&\quad +(k-1)(\ell-s)(2(m+\ell)+m+s+2(n-k)-1)\\
&\quad +(k-1)\ell(n-k-1)(2(m+\ell)+m+n-k-1)\\
&\quad +s(n-k-1)(2(m+s)+m+n-k-1)\big)I_s.
\end{split}
\end{equation*}

By taking into account the calculations made in \eqref{pipuntoE11} and \eqref{piE1jEj1} and replacing them in  \eqref{F2} and \eqref{F3} we
complete the proof of the lemma. \qed
\end{proof}

\begin{prop}\label{matrixE} For all $r>0$ and $0\le s\le\ell$ we have
\begin{align*}
(\tilde E\tilde H )_s (r)=& - \frac{(1+r^2)^2}{4}(m+\ell-s) \tilde h_s''(r) \\
&-\frac{(1+r^2)}{4r} (m+\ell-s)\big(2n-1+r^2-2r^2(m+\ell-s)\big)\tilde h_s'(r)\displaybreak[0]\\
&-\frac{(1+r^2)}{2r}s(\ell+k-s) \tilde h'_{s-1}(r) +\frac{(1+r^2)^2}{2r}(\ell-s)(n-k+s) \tilde h'_{s+1}(r) \displaybreak[0]\\
& +\frac{(1+r^2)}{r^2}  s(\ell+k-s)(k-s-m-n+1)\big(\tilde h_{s-1}(r)-\tilde h_s(r)\big) \displaybreak[0] \\
&+\frac{(1+r^2)}{r^2} (\ell-s)(n-k+s)(k-s-m-1)\big(\tilde h_{s+1}(r)-\tilde h_s(r)\big) \displaybreak[0]  \\
&+s(\ell+k-s)(2m+n+\ell-k)\big(\tilde h_{s-1}(r)-\tilde h_s(r)\big) .
\end{align*}
\end{prop}

\begin{proof}
We need to compute in each $M$-module $V_{\mathbf t(s)}$, $0\le s\le\ell$, the linear
transformations appearing in the differential operator $\tilde E$ in Theorem \ref{main''}.

It is important to recall, see \eqref{E1j}, that
$$\dot\pi(E_{1j})(V_{\mathbf t(s)})\subset V_{\mathbf t(s-1)} \quad \text{and} \quad
\dot\pi(E_{j1})(V_{\mathbf t(s)})\subset V_{\mathbf t(s+1)},$$
for $2\le j\le n$.

\noindent From the calculations made in the proof of Proposition \ref{matrixD} we have
\begin{align*}
\dot\pi(E_{11})v_s & =(m+\ell-s)v_s, \displaybreak[0] \\
\sum_{j=2}^n \dot\pi(E_{j1})\frac{d\tilde H}{dr}\dot\pi(E_{1j})v_s & =  \tilde h'_{s-1}\sum_{j=2}^n\dot\pi(E_{j1})\dot\pi(E_{1j})v_s=\tilde h'_{s-1} \, s(\ell+k-s)v_s,\displaybreak[0] \\
\sum_{j=2}^n \dot\pi(E_{1j})\frac{d\tilde H}{dr}\dot\pi(E_{j1})v_s & =  \tilde h'_{s+1}\sum_{j=2}^n\dot\pi(E_{1j})\dot\pi(E_{j1})v_s=\tilde h'_{s+1} \,(\ell-s)(n-k+s)v_s.
\end{align*}

\noindent It is easy to verify that
\begin{align*}
\sum_{j,k=2}^n  \big[\dot\pi  (E_{k1}), \big[ \dot\pi(E_{1j}), \tilde H\big]\big]\dot\pi(E_{jk}) &=
\big(\tilde h_{s}-\tilde h_{s-1}\big) \dot\pi(F)+\big(\tilde h_{s}-\tilde h_{s+1}\big) \dot\pi(F'),\\
\sum_{j,k=2}^n\big[\dot\pi(E_{1j}),\big[\dot\pi(E_{k1}),\tilde H\big]\big]\dot\pi(E_{jk}) & =\big(\tilde h_{s}-\tilde h_{s-1}\big) \dot\pi(F)+\big(\tilde h_{s}-\tilde h_{s+1}\big) \dot\pi(F').
\end{align*}

\noindent  Then  by using Lemma \ref{F7-F8} we get
\begin{align*}
\sum_{j,k=2}^n \Big(\big[\dot\pi & (E_{k1}),\big[ \dot\pi(E_{1j}), \tilde H\big]\big] + \big[\dot\pi(E_{1j}),\big[ \dot\pi(E_{k1}), \tilde H\big]\big] \Big) \dot\pi(E_{jk})\,  v_s \displaybreak[0]\\
&= 2\Big(\big(\tilde h_{s-1}-\tilde h_s\big)s(\ell+k-s)(k-s-m-n+1)\\
&\qquad \quad +\big(\tilde h_{s+1}-\tilde h_s\big)(\ell-s)(n-k+s)(k-s-m-1)\Big)v_s.
\end{align*}

\noindent On the other hand, we have
\begin{align*}
\sum_{k=1}^n \sum_{j=2}^n & \big[\dot\pi(E_{j1}) ,\tilde H \big]\dot\pi(E_{1k})\dot\pi(E_{kj})v_s\\ & = \sum_{j=2}^n \big[\dot\pi(E_{j1}),\tilde H \big]\dot\pi(E_{11})\dot\pi(E_{1j})v_s+
\sum_{j,k=2}^n \big[\dot\pi(E_{j1}),\tilde H \big]\dot\pi(E_{1k})\dot\pi(E_{kj})v_s\\
& = (m+\ell-s+1)\sum_{j=2}^n \big[\dot\pi(E_{j1}),\tilde H \big]\dot\pi(E_{1j})v_s + \big(\tilde h_{s-1}-\tilde h_s\big) \dot\pi(F)v_s,
\intertext{by using \eqref{pipuntoEj1} and Lemma \ref{F7-F8} we obtain}
&= \big(\tilde h_{s-1}-\tilde h_s\big) s (\ell+k-s) ( 2m+\ell+n-k).
\end{align*}

\noindent Now the proposition follows easily from Theorem \ref{main''}. \qed
\end{proof}

\section{Hypergeometrization}\label{Hipergeometrization}

\noindent Let us introduce the change of variables $t=(1+r^2)^{-1}$. Let $H(t)=\tilde H(r)$
and correspondingly put $h_s(t)=\tilde h_s(r)$.  Then the differential operator
of Proposition \ref{matrixD} becomes
\begin{equation*}
\begin{split}
(DH)_s(t) &=\;-\Big (t(1-t) h_{s}''(t) +\bigl(m+\ell-s+1-t(n+m+\ell-s+1)\bigr) h_{s}'(t) \\
&\qquad \quad +\frac 1{1-t}\,(\ell-s)(n+s-k)\bigl(h_{s+1}(t)-h_{s}(t)\bigr)\\
&\qquad \quad +\frac t{1-t}\, s(\ell-s+k)\bigl(h_{s-1}(t)-h_{s}(t)\bigr)\Big),
\end{split}
\end{equation*}
for $t\in(0,1)$ and $s=0,\dots,\ell$.
In matrix notation we have

\begin{equation}\label{operadorD1}
DH(t)= -\Big(t(1-t)H''(t) + (A_0-t(A_0+n) )H'(t)+\frac 1{1-t} (B_0+tB_1)H(t)\Big),
\end{equation}
where
\begin{equation}\label{A0B0B1}
\begin{split}
A_0 &=\sum_{s=0}^\ell (m+\ell-s+1)\, E_{ss},
\\
B_0&=\sum_{s=0}^\ell (\ell-s)(n+s-k) \left(E_{s,s+1}-E_{ss}\right), \\
B_1&=\sum_{s=0}^\ell s(\ell-s+k) \left(E_{s,s-1}-E_{ss}\right).
\end{split}
\end{equation}

Similarly, the differential operator $E$ of Proposition \ref{matrixE} becomes
\begin{equation*}
\begin{split}
(EH)_s(t) =-\bigg(& t(1-t)(m+\ell-s) h_{s}''(t) \\
& +(m+\ell-s)\big(m+\ell-s+1-t(m+\ell-s+n+1)\big) h_{s}'(t) \\
&+(\ell-s)(n-k+s)h_{s+1}'(t)-s(\ell+k-s)\,t\,h_{s-1}'(t)\\
&+\frac 1{1-t}\,(\ell-s)(n-k+s)(m+s-k+1)\bigl(h_{s+1}(t)-h_{s}(t)\bigr)\\
&+\frac
1{1-t}\,s(\ell-s+k)(m+n+s-k-1)\bigl(h_{s-1}(t)-h_{s}(t)\bigr) \\
& - s(\ell-s+k)(2m+n+\ell-k)\bigl(h_{s-1}(t)-h_{s}(t)\bigr)\bigg),
\end{split}
\end{equation*}
for $t\in(0,1)$ and $s=0,\dots,\ell$. In matrix notation we have

\begin{equation}\label{operadorE1}
EH(t)= -\Big(t(1-t)MH''(t) + (C_0-tC_1 )H'(t)+\frac 1{1-t}
(D_0+tD_1)H(t)\Big),
\end{equation}
where
\begin{align*}
M =&\;\sum_{s=0}^\ell(m+\ell-s)\, E_{ss}, \displaybreak[0]\\
C_0 =&\;\sum_{s=0}^\ell (m+\ell-s)(m+\ell-s+1)\, E_{ss}+\sum_{s=0}^\ell (\ell-s)(n-k+s)\, E_{s,s+1}, \displaybreak[0]\\
C_1=&\;\sum_{s=0}^\ell (m+\ell-s)(m+\ell-s+n+1)\, E_{ss}+ \sum_{s=0}^\ell s(\ell+k-s)E_{s,s-1}, \displaybreak[0]\\
D_0=&\;\sum_{s=0}^\ell(\ell-s)(n-k+s)(m+s-k+1)\left(E_{s,s+1}-E_{ss}\right)\\
&-\sum_{s=0}^\ell s(\ell+k-s)(m+\ell-s+1)\left(E_{s,s-1}-E_{ss}\right),\displaybreak[0]\\
D_1=&\;\sum_{s=0}^\ell s(\ell-s+k)(2m+\ell+n-k)\left(E_{s,s-1}-E_{ss}\right).
\end{align*}

\

Let us recall that if $\Phi$ is an irreducible spherical function of one-step $K$-type ${\kk}$ as in
\eqref{onestep1}, then  the associated function $H(t)=(h_0(t),\dots,h_\ell(t))^t$, $0<t<1$
is an eigenfunction of the differential operators $D$ and $E$.

\begin{thm}\label{asintotica}
Let $H(t)=(h_0(t),\dots,h_\ell(t))^t$,  be the function associated to an irreducible spherical
function $\Phi$ of one-step $K$-type $\kk$. Then $H$ is a polynomial eigenfunction of the differential operators $D$ and $E$ and $H(1)=(1,\dots,1)^t$.\\
 Moreover if $ m+\ell+1\le s\le\ell$, the function $h_s$ is of the form
\begin{equation}\label{cero}
h_s(t)=t^{s-m-\ell}g_s(t),
\end{equation}
with $g_s$ polynomial and $g(0)\neq 0$.
\end{thm}
\begin{proof}
On the open subset ${\A}$ of $G$ defined by the
condition $g_{n+1,n+1}\ne0$ we put $H(g)=\Phi(g)\Phi_\kk(g)^{-1}$.
For any $-\pi/2<\theta<\pi/2$ let us consider the elements
$$a(\theta)= \left(\begin{matrix} \cos\theta& 0& \sin\theta\\ 0&I_{n-1}&0\\ -\sin\theta& 0&\cos
\theta\end{matrix}\right)\in\A$$
and the left upper $n\times n$ corner $A(\theta)$ of $a(\theta)$.

Let $\pi$ denotes the irreducible finite-dimensional representation of $\U(n)$ of highest weight $\kk$.  Since  $A(\theta)$ commutes with $\U(n-1)$ then $\pi(A(\theta))$ is a scalar in each  $\U(n-1)$-submodule  $V_{{\bf t}(s)}$ of $V_\kk$,
see \eqref{onestep2}.
Thus a highest weight vector of $V_{{\bf t}(s)}$ as $\U(n-1)$-module is
a vector of the $\U(n)$-module $V_{\kk}$ of weight $\mu_s$, see \eqref{mus}.
Then $\Phi_\kk(a(\theta))=\pi(A(\theta))$ in each
$V_{{\bf t}(s)}$ is equal to $\cos(\theta)^{m+\ell-s}$ times the identity.
Also $\Phi(a(\theta))$ is a scalar $\phi_s(a(\theta))$ in each $V_{{\bf t}(s)}$.

The projection of $a(\theta)$ into $P_n(\CC)$ is $p(a(\theta))=(\tan(\theta),0,\dots,0,1))$. Thus if we make the change of
variables $r=\tan(\theta)$ and $t=(1+r^2)^{-1}=\cos^2(\theta)$ we have
\begin{equation}\label{asymptotic}
\phi_s(a(\theta))=t^{(m+\ell-s)/2}h_s(t).
\end{equation}

In [PT4] it is proved that $\phi_s(a(\theta))=(\cos\theta)^{m+\ell-s}p_s(\sin^2\theta)$ where $p_s$ is a polynomial.
Therefore $h_s(a(\theta))=p_s(\sin^2\theta)$ and $h_s(t)=p_s(1-t)$. Thus $H=H(t)$ is polynomial, and it only remains to prove
\eqref{cero}.

By taking limit when $t\rightarrow 0$ in \eqref{asymptotic} we get
$$\lim_{t\rightarrow 0}t^{(m+\ell-s)/2}h_s(t)=\phi_s(a(\pi/2)).$$
If $x\in\CC$ with $\vert x\vert=1$, then
$$b(x)=\left(\begin{matrix} x& 0& 0\\
0&I_{n-1}&0\\
0& 0&x^{-1}\end{matrix}\right)\in K$$
has the following nice property: $b(x)a(\pi/2)=a(\pi/2)b(x^{-1})$.
Hence
\begin{align*}
\phi_s(b(x)a(\pi/2))&=x^{m+\ell-s}\phi_s(a(\pi/2)),
\\ \phi_s(a(\pi/2)b(x^{-1}))&=
\phi_s(a(\pi/2))x^{-(m+\ell-s)}.
\end{align*}
Therefore if $s\ne m+\ell$ then
$\phi_s(a(\pi/2))=0$, and we have obtained
\begin{enumerate}
  \item[ (i)] if $s\ne m+\ell$, $\lim_{t\rightarrow0} t^{(m+\ell-s)/2}h_s(t)=0$,

\item[ (ii)] if $s=m+\ell$, $\lim_{t\rightarrow0} t^{(m+\ell-s)/2}h_s(t)=\phi_s(a(\pi/2))$.
\end{enumerate}
In particular if $m+\ell+1\le s\le\ell$ we have
$$\lim_{t\rightarrow 0}t^{(m+\ell-s)/2}h_s(t)=0. $$
Hence $t=0$ is a zero of $h_s$ of order $k\geq 1$.

If $H(t)=\sum_{j}t^jH_j$ with column vector coefficients $H_j=(H_{0j},\dots,H_{\ell,j})^t$, then
from $DH=\lambda H$ we get that the following three term recursion relation holds for all $j$,
\begin{equation}\label{recursion}
\begin{split}
\big((j&-1)(j-2)+(j-1)(A_0+n)+B_1-\lambda\big)H_{j-1}\\
&-\big(2j(j-1)+j(2A_0+n)-B_0-\lambda\big)H_j+(j+1)(j+A_0)H_{j+1}=0.
\end{split}
\end{equation}
Therefore if $m+\ell+1\le s\le\ell$ we have $h_s(t)=\sum_{j\ge k}t^jH_{s,j}$, with $H_{s,k}\ne0$. If we put $j=k-1$
from \eqref{recursion} we obtain
$$k(k+m+\ell-s)H_{s,k}=0.$$
Hence $k=s-m-\ell$ which completes the proof of the theorem.
\qed
\end{proof}

\medskip
A particular class of matrix-valued differential operators are the
hypergeometric ones which are of the form
\begin{equation}\label{ophiper}
t(1-t) \frac{d^2}{dt^2}+(C-tU)\frac{d}{dt}-V
\end{equation}
where $C,U$ and $V$ are constant square matrices. These were
introduced in \cite{T2}. We observe that the differential operator
appearing in \eqref{operadorD1}, obtained from the Casimir of $G$,
is closed to be hypergeometric but it is not.

\begin{defn}
A differential operator $D$ is conjugated to a hypergeometric
operator if there exists an invertible smooth matrix-valued function
$\Psi=\Psi(t)$ such that the differential operator $\tilde D$
defined by $\tilde DF= \Psi^{-1}D(\Psi F)$ is of the form
\begin{equation*}
{\tilde D}=t(1-t)\frac{d^2}{dt^2}+{\tilde
A}_1(t)\frac{d}{dt}+{\tilde A}_0,
\end{equation*}
where $\tilde A_1(t)$ is a matrix polynomial of degree 1 and
$\tilde A_0$ is a constant matrix.
\end{defn}

The differential operator
$$-D=t(1-t)\frac{d^2}{dt^2} + (A_0-tA_1 )\frac{d}{dt}+\frac 1{1-t}
(B_0+tB_1)H(t)$$ is conjugated to a hypergeometric operator if and
only if there exists an invertible smooth matrix-valued function
$\Psi(t)$ and constant matrices $C, U, V,$ such that
\begin{equation}\label{ecuacion01}
2t(1-t)\Psi(t)^{-1} \Psi'(t)+ \Psi(t)^{-1}(A_0-t(A_0+n))\Psi(t)=C-tU
\end{equation}
  and
\begin{equation}\label{ecuacion02}
\begin{split}
  t(1-t)\Psi(t)^{-1}\Psi''(t)&+\Psi(t)^{-1}(A_0-t(A_0+n))\Psi'(t)\\
&+ \frac{1}{1-t}\Psi(t)^{-1}(B_0+tB_1)\Psi(t)=-V.
\end{split}
\end{equation}

A problem of this kind was first solved in \cite{RT}.
Motivated by the form of the solution found there, in \cite{PT2} a solution of \eqref{ecuacion01} and \eqref{ecuacion02} is
looked for among those functions of the form $\Psi(t)=XT(t)$, where $X$
is a constant lower triangular matrix with ones in the main diagonal
and $T(t)=\sum_{0\le s\le\ell}(1-t)^sE_{ss}$. Then it is proved that
$X$, equal to the Pascal matrix, provides the unique solution to
equations \eqref{ecuacion01} and \eqref{ecuacion02}, and the
following fact is established in Corollary 3.3 in \cite{PT2}.

\begin{prop}\label{D}
Let $T(t)=\sum_{i=0}^\ell (1-t)^i E_{ii}$, let $X$ be the Pascal
matrix $X_{ij}=\binom ij$,  and let $\Psi(t)=XT(t)$. Then $\tilde
DF(t)=\Psi^{-1}(t)(-D)\big(\Psi(t)F(t)\big)$ is a hypergeometric
operator of the form
$$\tilde DF(t)=t(1-t)F''(t)+(C-tU)F'(t)-VF(t)$$
with
\begin{align*}
C& =\sum_{s=0}^\ell (m+\ell-s+1)E_{ss}-\sum_{s=1}^\ell sE_{s,s-1},
\quad U=\sum_{s=0}^\ell (n+m+\ell+s+1) E_{ss} , \displaybreak[0]\\
  V&=  \sum_{s=0}^\ell s(n+m+s-k)E_{ss}- \sum_{s=0}^{\ell-1}
(\ell-s)(n+s-k)E_{s,s+1}.
\end{align*}
\end{prop}

\begin{prop}\label{E}
Let $\tilde E$ be defined by $\tilde EF(t)=
\Psi(t)^{-1}(-E)\big(\Psi(t)F(t)\big)$. Then
$$\tilde EF(t)=t(M_0-tM_1)F''(t)+(P_0-tP_1)F'(t)-(m-k)VF(t)$$
with
\begin{align*}
M_0=&\;\sum_{s=0}^\ell (m+\ell-s)E_{ss}-\sum_{s=1}^\ell sE_{s,s-1},
\quad M_1=\sum_{s=0}^\ell (m+\ell-s)E_{ss} , \displaybreak[0]\\
P_0=&\;\sum_{s=0}^\ell
\big((m+\ell)(m+\ell+1)+\ell(n-k)-2s(n+m-k+s)\big)E_{ss}\\
&-\sum_{s=1}^{\ell}s(n-k+\ell+2m+s)E_{s,s-1}+\sum_{s=0}^{\ell-1}(\ell-s)(n-k+s)E_{s,s+1},\displaybreak[0]\\
P_1=&\;\sum_{s=0}^\ell(m+\ell-s)(m+n+\ell+s+1)E_{ss}+\sum_{s=0}^{\ell-1}(\ell-s)(n-k+s)E_{s,s+1}.
\end{align*}
\end{prop}

\begin{remark}
A proof of the above proposition can be reached by a long, direct
and careful computation. A shorter path would be as follows: to start by
observing that the differential operator $E^{PR}$ obtained from the
differential operator $E$ which appears in Section 3 of \cite{PR},
by making the change of variables $t=1-u$, it is given in terms of
our $\tilde D$ and $\tilde E$ by $E^{PR}=2(2m+\ell)\tilde D-3\tilde E$.
\end{remark}

\section{The eigenvalues of $D$ and $E$}\label{autovalores}

\noindent According to Proposition \ref{defeq} and Lemma \ref{Kinvariantes} the irreducible spherical
functions $\Phi$ of our pair $(G,K)$ are eigenfunctions of the differential operators  $\Delta_P$ and $\Delta_Q$, introduced in \eqref{DeltaPQ}.
In this situation the eigenvalues $[\Delta_P\Phi](e)$ and $[\Delta_Q\Phi](e)$ are scalars  because the
irreducible spherical functions are all of height one.

Let $\Phi=\Phi^{\mm,\kk}$ be the irreducible spherical function associated to the representations $\mm$ of $G$ and $\kk$ of $K$,
$$ \kk=(\underbrace{m+\ell,\dots,m+\ell}_k,\underbrace{ m,\dots,m}_{n-k}),$$
$$ \mm=\mm(w,r)=(w+m+\ell,\underbrace{m+\ell,\dots,m+\ell}_{k-1},m+r,\underbrace{m,\dots,m}_{n-k-1},-w-r ),$$
where  $0\le w$, $0\le r\le \ell$ and $ -m-r\leq w$. See \eqref{onestep1} and \eqref{n+1tupla}. Let us recall that $\Phi$ is an spherical function of type $\kk$, because $s_\mm=s_\kk$.

\smallskip
 Since the differential operators $D$ and $E$ come from $\Delta_P$ and $\Delta_Q$ respectively, (see \eqref{DeltaHPhi}) the corresponding function $H(t)=H^{\mm,\kk}(t)$  satisfies
$$DH^{\mm,\kk}=[\Delta_P\Phi^{\mm,\kk}](e)H^{\mm,\kk},\qquad EH^{\mm,\kk}=[\Delta_Q\Phi^{\mm,\kk}](e)H^{\mm,\kk}.$$

\noindent Now we will concentrate on computing the scalars  $[\Delta_P\Phi^{\mm,\kk}](e)$ and
$[\Delta_Q\Phi^{\mm,\kk}](e)$. We start by observing that if $\Delta\in D(G)^K$, then we have that
$$[\Delta \Phi^{\mm,\kk}](e)=\dot\pi_{\mm}(\Delta).$$

Let us recall that the Casimir operator of $\GL(n,\CC)$  is
$$\Delta^{(n)}=\sum_{1\le i,j\le n}E_{ij}E_{ji}=\sum_{1\leq i\leq n} E_{ii}^2+ \sum_{1\leq i<j\leq n} (E_{ii}-E_{jj})+ 2\sum_{1\leq
i<j\leq n} E_{ji}E_{ij}.$$
Thus the differential operator $\Delta_P=\sum_{1\le j\le n}E_{n+1,j}E_{j,n+1}$,  can be also
written as
$$\Delta_P=\frac12
\Bigl(\Delta^{(n+1)}-\Delta^{(n)} -E_{n+1,n+1}^2- \sum_{1\leq j\leq
n} (E_{jj}-E_{n+1,n+1}) \Bigr).$$

The highest weight of $V_\mm$ is
\begin{align*}
\mu_\mm=&\;(w+m+\ell)x_1+(m+\ell)(x_2+\cdots+x_k)+(m+r)x_{k+1}\\
&\quad +m(x_{k+2}+\cdots+x_n)-(w+r)x_{n+1},
\end{align*}
then it follows that
\begin{align*}
 \dot\pi_\mm & ( \Delta^{(n+1)})=
\bigl((w+m+\ell)^2+(m+\ell)^2(k-1)+(m+r)^2+m^2(n-k-1) \displaybreak[0]\\
&+(w+r)^2+w(k-1)+w+\ell-r+(w+\ell)(n-k-1) +2w+m+\ell+r \displaybreak[0]\\
&+(\ell-r)(k-1)+\ell(k-1)(n-k-1) +(m+\ell+w+r)(k-1)\displaybreak[0]\\
&+r(n-k-1)+m+w+2r +(m+w+r)(n-k-1)\bigr)I_\mm.
\end{align*}

\noindent
If $v_0\in V_\kk\subset V_\mm$ is a $\liegl(n,\CC)$-highest weight vector, then its weight is
$$\mu_0=(m+\ell)(x_1+\cdots+x_{k})+m(x_{k+1}+\cdots + x_n),$$
because
$$\dot\pi_\mm(E_{n+1,n+1})v_0 =\dot\pi_\mm\big(\sum_{1\leq j\leq n+1}\!\!\!\!E_{jj}\big)v_0 - \dot\pi_\mm\big(\sum_{1\leq j\leq n}\!\!\!E_{jj}\big)v_0 =(s_\mm-s_\kk)v_0=0.$$
 We also have that
\begin{align*}
\dot\pi_\mm(\Delta^{(n)})v_0 &=\bigl(k(m+\ell)^2+(n-k)m^2+k(n-k)\ell \bigr)v_0
\end{align*}
and therefore we obtain
$$\dot\pi_\mm(\Delta_P)v_0=\big(w(w+m+\ell+r+n)+r(m+r-k+n)\big)v_0.$$

\

To compute
 $\dot\pi_\mm(\Delta_Q) $ we start by writing $\Delta_Q=\sum_{1\leq i,j\leq n} E_{n+1,i}E_{j,n+1}E_{ij}$ in terms of elements in $D(G)^G$ and $D(K)^K$.\\

Let us recall that in the proof of Lemma \ref{F7-F8} we introduced the element
$$F^{(n)}=\sum_{1\le i,j,k\le n}E_{ki}E_{ij}E_{jk}$$
which is in the center of the universal enveloping algebra of
${\lieg}(n,\CC)$. Hence
\begin{equation*}\label{E0}
\begin{split}
F^{(n+1)}-F^{(n)}=&\;\sum_{1\leq j,k \leq n+1}
E_{k,n+1}E_{n+1,j}E_{jk}+\sum_{\substack{ 1\leq i \leq n\\ 1\leq k\leq n+1}}
E_{ki}E_{i,n+1}E_{n+1,k}\\
&+\sum_{1\leq i,j \leq n}E_{n+1,i}E_{ij}E_{j,n+1}.
\end{split}
\end{equation*}

\medskip
We also have
\begin{align*}
\sum_{1\leq j,k \leq n+1} E_{k,n+1}E_{n+1,j}E_{jk}& =\;\Delta^{(n+1)}+\Delta_Q+E_{n+1,n+1}^3-(n+1)E_{n+1,n+1}^2\\
& \quad +(2E_{n+1,n+1}-1)\Delta_P,\displaybreak[0]\\
\sum_{\substack{1\leq i \leq n\\ 1\leq k\leq n+1}}
E_{ki}E_{i,n+1}E_{n+1,k}&=\;\Delta^{(n)}-E_{n+1,n+1}\sum_{1\leq i \leq n}E_{ii}+\Delta_Q+E_{n+1,n+1}\Delta_P,\displaybreak[0]\\
\sum_{1\leq i,j \leq n} E_{n+1,i}E_{ij}E_{j,n+1}& =\;n\Delta_P+\Delta_Q.
\end{align*}

\noindent Thus
\begin{equation*}
\begin{split}
\Delta_Q=\;\frac13\Big( & F^{(n+1)}-F^{(n)}-\Delta^{(n+1)}-\Delta^{(n)}-(3E_{n+1,n+1}+n-1)\Delta_P\\
&-(n+1)E_{n+1,n+1}^2+E_{n+1,n+1}\sum_{1\le i\le n}E_{ii}\Big).
\end{split}
\end{equation*}

\noindent Taking  into account that $E_{n+1,n+1}v_0=0$ we have
\begin{equation}\label{E1}
\begin{split}
\dot\pi_\mm(\Delta_Q)v_0=
&\;\frac13\Big(\dot\pi_\mm\big(F^{(n+1)}\big)-\dot\pi_\mm\big(F^{(n)}\big)-\dot\pi_\mm\big(\Delta^{(n+1)}\big)-\dot\pi_\mm\big(\Delta^{(n)}\big)\\
&-(n-1)\dot\pi_\mm\big(\Delta_P\big)\Big)v_0.
\end{split}
\end{equation}
Observe that we have already calculated   $\dot\pi_\mm\big(\Delta^{(n+1)}\big)v_0$, $\dot\pi_\mm\big(\Delta^{(n)}\big)v_0$, $\dot\pi_\mm\big(\Delta_P\big)v_0$ and also $ \dot\pi_\mm \big(F^{(n)}\big)v_0$  in \eqref{F6}.

Since $F^{(n+1)}$ is $\GL(n+1,\CC)$-invariant it is enough to compute
$\dot\pi_\mm (F^{(n+1)})$ on a highest weight vector $v_\mm$ of $V_\mm$; hence it is enough to write $F^{(n+1)}$ modulo the left ideal of $U(\liegl(n+1,\CC))$ generated by $\{E_{ij}:1\le i<j\le n+1\}$.

To rewrite the summand  $E_{ki}E_{ij}E_{jk}$ of $F^{(n+1)}$ we partition the set of indices as we did in \eqref{partition}. This yields to
\begin{equation*}
F^{(n+1)}\equiv  \sum_{1\le k\le n+1}E_{kk}^3+\sum_{1\le i<j\le
n+1}(E_{ii}-E_{jj})(2E_{ii}+E_{jj}+2(n+1-j)+1).
\end{equation*}

\noindent Thus we obtain
\begin{align*}
 \dot\pi_\mm  (F^{(n+1)}) & =\bigl((w+m+\ell)^3+(m+\ell)^3(k-1)+(m+r)^3+m^3(n-k-1)\displaybreak[0]\\
&-(w+r)^3+w(k-1)(2w+3m+3\ell+2n-k+1) \displaybreak[0]\\
&+(w+\ell-r)(2w+3m+2\ell+r+2n-2k+1)\displaybreak[0]\\
&+(w+\ell)(n-k-1)(2w+3m+2\ell+n-k+1)\displaybreak[0]\\
&+(2w+m+\ell+r)(w+2m+2\ell-r+1)\displaybreak[0]\\
&+(\ell-r)(k-1)(3m+2\ell+r+2n-2k+1)\displaybreak[0]\\
&+\ell(k-1)(n-k-1)(3m+2\ell+n-k+1)\displaybreak[0]\\
&+(m+\ell+w+r)(k-1)(2m+2\ell-w-r+1)\displaybreak[0]\\
&+r(n-k-1)(3m+2r+n+1-k)\displaybreak[0]\\
&+(m+w+2r)(2m+r-w+1)\displaybreak[0]\\
&+(m+w+r)(n-k-1)(2m-w-r+1)\bigr)I_\mm.
\end{align*}

From \eqref{E1}, taking into account the previous calculations, we obtain after some careful computation
$$\dot\pi_\mm(\Delta_Q)v_0= \big(w ( m+\ell - r) (w + m +\ell + r + n)+r (m - k) (m + r - k + n)\big)v_0.$$

\smallskip
Therefore we have proved the following proposition.

\begin{prop}\label{autovalores1} The function $H^{\mm,\kk}$ associated to the spherical function $\Phi^{\mm,\kk}$ satisfies
$$DH^{\mm,\kk}=\big(w(w+m+\ell+r+n)+r(m+r-k+n)\big)H^{\mm,\kk},$$
$$EH^{\mm,\kk}=\big(w (m+\ell - r) (w + m + \ell + r + n)+r (m - k) (m + r - k + n)\big)H^{\mm,\kk}.$$
\end{prop}

\smallskip
\section{The one-step spherical functions}

\subsection{The simultaneous eigenfunctions of $D$ and $E$}
\mbox{} \\
In Section \ref{Hipergeometrization} we introduced the differential operators $\tilde D=\Psi^{-1}(-D)\Psi$ and $\tilde E=\Psi^{-1}(-E)\Psi$.
If we put $F^{\mm,\kk}(t)=\Psi(t)H^{\mm,\kk}(t)$, from Proposition \ref{autovalores1} we get
\begin{equation}\label{autovalores2}
\tilde D F^{\mm,\kk}=\lambda(w,r)F^{\mm,\kk},\qquad \tilde E F^{\mm,\kk}=\mu(w,r)F^{\mm,\kk},
\end{equation}
where
\begin{equation}\label{autovalores3}
\begin{split}
\lambda(w,r)&=-w(w+m+\ell+r+n)-r(m+r-k+n),\\
\mu(w,r)&=-w (m+\ell - r) (w + m +\ell + r + n)-r (m - k) (m + r - k + n).
\end{split}
\end{equation}

Since spherical functions are analytic functions on  $G$ it follows
that $F^{\mm,\kk}(t)$ are analytic in a neighborhood  of
$t=1$. Thus we make the change of variables $u=1-t$, so we will be
dealing with analytic functions around $u=0$ which are simultaneous
eigenfunctions of the following differential operators,
\begin{equation}\label{DD,EE}
\begin{split}
DF&=u(1-u)F''+(U-C-uU)F'-VF,\\
EF&=(1-u)(M_0-M_1+uM_1)F''+(P_1-P_0-uP_1)F'-(m-k)VF,
\end{split}
\end{equation}
where the coefficient matrices are those given in Propositions \ref{D} and \ref{E}.

\medskip
The matrix-valued hypergeometric function was introduced  in
\cite{T2}. Let $W$ be a finite-dimensional complex vector space, and
let $A,B$ and $C\in \End(W)$.

\noindent The hypergeometric equation is
\begin{align}\label{hiper0}
u(1-u)F'' +(C-u (I+A+B))F'- AB F=0.
\end{align}
If the eigenvalues of $C$ are not in $-\NN_0$, we define the function
\begin{equation*}
{}_2\!F_1 \!\!\left( \begin{smallmatrix} A\,;\,B\\
C\end{smallmatrix} ; u\right)=\sum_{m=0}^\infty
\frac{u^m}{m!}(C;A;B)_m ,
\end{equation*}
where the symbol $(C;A;B)_m$ is defined inductively by $(C;A;B)_0
=I$ and
\begin{align*}
(C;A;B)_{m+1}=(C+m)^{-1}(A+m)(B+m)(C;A;B)_m,
\end{align*}
for all $m\ge0$.
The function ${}_2\!F_1 \!\!\left( \begin{smallmatrix} A\,;\,B\\
C\end{smallmatrix} ; u\right)$ is analytic on $|u|<1$ with values in
$\End(W)$. Moreover if $F_0\in W$, then $F(u)= {}_2\!F_1 \!\!\left( \begin{smallmatrix} A\,;\,B\\
C\end{smallmatrix} ; u\right)\!F_0$ is a solution of the
hypergeometric equation \eqref{hiper0} such that $F(0)=F_0$.
Conversely any  solution $F$, analytic at $u=0$ is of this form.

In the scalar case the differential operator \eqref{ophiper} is
always of the form given in \eqref{hiper0}; after solving a
quadratic equation we can find $A,B$ such that $U=1+A+B$ and $V=AB$.
This is not necessarily the case when $\dim(W)>1$. In other words, a
differential equation of the form
\begin{equation}\label{hyper1}
u(1-u)F''+(C-u U)F'- V F=0,
\end{equation}
with $U,V,C\in\End(W)$, cannot always be reduced to the form
of \eqref{hiper0}, because a quadratic equation in a noncommutative
setting as $\End(W)$ may have no solutions. Thus it is also
important to give a way to solve \eqref{hyper1}.

If the eigenvalues of $C$ are not in $-\NN_0$, let us introduce the
sequence $[C,U,V]_m\in\End(W)$ by defining inductively $[C;U;V]_0=I$
and
\begin{equation*}
[C;U;V]_{m+1}=(C+m)^{-1}(m^2+m(U-1)+V)[C;U;V]_m,
\end{equation*}
for all $m\ge0$. Then the function
\begin{equation*}
{}_2\!H_1 \!\!\left( \begin{smallmatrix} U\,;\,V\\
C\end{smallmatrix} ; u\right)=\sum_{m=0}^\infty
\frac{u^m}{m!}[C;U;V]_m ,
\end{equation*}
is analytic on $|u|<1$ and it is the unique solution of
\eqref{hyper1} analytic at $u=0$, with values in $\End(W)$, whose
value at $u=0$ is $I$.

\smallskip

Let $F^{\mm,\kk}$ be the function defined by
$F^{\mm,\kk}(u)=F^{\mm,\kk}(1-t)$. Then the map
$\Phi^{\mm,\kk}\mapsto F^{\mm,\kk}$ establishes a one-to-one
correspondence between the equivalence classes of irreducible
spherical functions of $\SU(n+1)$ of a fixed $K$-type $\kk$ as in
\eqref{onestep1}, with certain simultaneous $\CC^{\ell+1}$-analytic
eigenfunctions on the open unit disc of the differential operators
$D$ and $E$ given in \eqref{DD,EE}.

Therefore
$$F^{\mm,\kk}(u)={}_2\!H_1 \!\!\left( \begin{smallmatrix} U\,;\,V+\lambda\\
U-C\end{smallmatrix} ; u\right)F^{\mm,\kk}(0),$$ with
$\lambda=\lambda(w,r)$.

The goal now is to describe all simultaneous $\CC^{\ell+1}$-valued
eigenfunctions of the differential operators $D$ and $E$ analytic on
the open unit disc $\Omega$. We let
$$V_{\lambda}=\{F=F(u):DF=\lambda F, \;F\; \text{analytic on}\; \Omega\}.$$

Since the initial value $F(0)$ determines $F\in V_\lambda$, we have
that the linear map $\nu:V_\lambda\rightarrow \CC^{\ell+1}$ defined
by $\nu(F)=F(0)$ is a surjective isomorphism. Because $\Delta_P$ and
$\Delta_Q$ commute, both being in the algebra $D(G)^K$ which is
commutative, $D$ and $E$ also commute. Moreover, since $E$ has
polynomial coefficients, $E$ restricts to a linear operator of
$V_\lambda$. Thus we have the following commutative diagram:
\begin{equation*}\label{diagrama1}
\begin{CD}
V_\lambda @ >E >>V_\lambda \\ @ V \nu VV @ VV \nu V \\
\CC^{\ell+1} @> M(\lambda) >> \CC^{\ell+1}
\end{CD}
\end{equation*}
where $M(\lambda)$ is the $(\ell+1)\times (\ell+1)$ matrix given by
\begin{equation*}\label{Mlambda}
\begin{split}
M(\lambda)=&\;(M_0-M_1)(U-C+1)^{-1}(U+V+\lambda)(U-C)^{-1}(V+\lambda)\\
&+(P_1-P_0)(U-C)^{-1}(V+\lambda)-(m-k)V.
\end{split}
\end{equation*}

\smallskip
Although the matrix $M(\lambda)$  has a complicated form we are able
to know its characteristic polynomial, via an indirect argument (cf.
Theorem 10.3 in \cite{GPT1}).

\begin{prop}\label{caracteristico}
$$\det(\mu-M(\lambda))=\prod_{r=0}^\ell(\mu-\mu_r(\lambda)),$$
where $\mu_r(\lambda)=\lambda(m+\ell-r)+r(m+r-k+n)(\ell-r+k)$.
Moreover, all eigenvalues $\mu_k(\lambda)$ of $M(\lambda)$ have
geometric multiplicity one. In other words all eigenspaces are one-dimensional.
If $v=(v_0,...,v_\ell)^t$ is a nonzero
$\mu_r(\lambda)$-eigenvector of $M(\lambda)$, then $v_0\neq 0$.
\end{prop}
\begin{proof}
Let us consider the polynomial $p\in\CC[\lambda,\mu]$ defined by
$p(\lambda,\mu)=\det(\mu-M(\lambda))$. For each integer $r$ such
that $0\le r\le\ell$ let $\lambda(w,r)=-w(w+m+\ell+r+n)-r(m+r-k+n)$.
Then from \eqref{autovalores2} and \eqref{autovalores3} we
have
$$p(\lambda(w,r),\mu_r(\lambda(w,r))=0,$$ for all $w\in\ZZ$
such that $0\le w$ and $0\le w+m+r$. Since there are infinitely many
such $w$, the polynomial function $w\mapsto
p(\lambda(w,r),\mu_r(\lambda(w,r)))$ is identically zero on $\CC$.
Hence, given $r$ $(0\le r\le\ell)$, we have
$p(\lambda,\mu_r(\lambda))=0$ for all $\lambda\in\CC$.

Now $\Lambda=\{\lambda\in\CC:\mu_r(\lambda)=\mu_{r'}(\lambda) \text{
for some } 0\le r<r'\le\ell\}$ is a finite set, in fact
$\mid\Lambda\mid\le\ell(\ell+1)/2$. Since for any $\lambda$,
$p(\lambda,\mu)$ is a monic polynomial in $\mu$ of degree $\ell+1$,
it follows that if $\lambda\in\CC-\Lambda$, then
\begin{equation}\label{p}
p(\lambda,\mu)=\prod_{r=0}^\ell(\mu-\mu_r(\lambda))
\end{equation}
for all $\mu\in\CC$.
Now it is clear that \eqref{p} holds for all $\lambda$ and all
$\mu$, which completes the proof of the first assertion.

To prove the last two statements of the proposition we point out
that the matrix $M(\lambda)=A+B$, where $A$ is a lower triangular
matrix and
$$B=\sum_{s=0}^{\ell-1}
\frac{(n+s-1)(n+s+\ell)(s+k)}{(n+2s-1)(n+2s)}E_{s,s+1}.$$ Therefore
if $v=(v_0,\dots,v_\ell)^t$ is a $\mu$-eigenvector of $M(\lambda)$,
then $v_0$ determines $v$, because the coefficients of $B$ are not
zero, and this implies that the geometric multiplicity of $\mu$ is
one and that $v_0\ne0$. \qed
\end{proof}

\medskip
In this way we have proved the following theorem.

\begin{thm}
The simultaneous  $\CC^{\ell+1}$-valued eigenfunctions of the
differential operators $D$ and $E$ analytic on the open unit disc $\Omega$ are the
scalar multiples of
\begin{equation*}
F_r(u)={}_2\!H_1 \!\left( \begin{smallmatrix} U\,;\,V+\lambda\\
U-C\end{smallmatrix} ; u\right)F_{r}(0),
\end{equation*}
where  $F_{r}(0)=(v_0,...,v_\ell)^t$ is the unique $\mu_r(\lambda)$-eigenvector of
$M(\lambda)$ normalized by
$v_0=1$, for some $0\le r\le\ell$ and $\lambda\in\CC$. Notice that
$$DF_r=\lambda F_r\quad \text{  and } \quad EF_r=\mu_r(\lambda)F_r.$$
\end{thm}

\subsection{Spherical functions}
\mbox{} \\
We recall that we were interested to determine the irreducible spherical functions of $G=\SU(n+1)$ whose $K$-type ($K=\U(n)$) is of the form
$$ \kk=(\underbrace{m+\ell,\dots,m+\ell}_k,\underbrace{ m,\dots,m}_{n-k}),$$
by giving  an explicit expression of their restriction to the one-dimensional abelian subgroup $A$ introduced in \eqref{atheta}.
At this point we have established the following characterization.

\begin{thm}\label{caracterizacion}
There is a bijective correspondence between the equivalence classes
of all irreducible spherical functions of  $\SU(n+1)$ of a $K$-type
$\kk$ as in \eqref{onestep1}, and the set of pairs
$(\lambda(w,r),\mu(w,r))\in\CC\times\CC$ where
\begin{equation*}
\begin{split}
\lambda(w,r)&=-w(w+m+\ell+r+n)-r(m+r-k+n),\\
\mu(w,r)&=-w (m+\ell - r) (w + m + \ell + r + n)-r (m - k) (m + r - k +
n),
\end{split}
\end{equation*}
with  $w$ a nonnegative integer, $0\le r\le\ell$ and $0\le w+m+r$.
In particular $\Phi^{\mm(w,r),\kk}$ can be obtained explicitly from
the following vector-valued function
\begin{equation*}\label{final2}
F_{w,r}(u)={}_2\!H_1 \left( \begin{smallmatrix} U\,;\,V+\lambda(w,r)\\
U-C\end{smallmatrix} ; u\right)F_{w,r}(0),
\end{equation*}
where $F_{w,r}(0)$ is the unique $\mu(w,r)$-eigenvector of
$M(\lambda(w,r))$ such that $F_{w,r}(0)=(1, v_1,\dots,v_\ell)^t$.
\end{thm}

From the above theorem it follows that the pair of eigenvalues
$[\Delta_P\Phi](e)$ and $[\Delta_Q\Phi](e)$ characterize the irreducible
spherical functions of type $\kk$ as in \eqref{onestep1}. This suggests that
one should be able to prove that the algebra $D(G)^K$ is generated by
$\Delta_P$ and $\Delta_Q$ modulo the two-sided ideal generated by
the kernel in $D(K)$ of the representation $\dot\pi$ of highest weight $\kk$.

Another consequence is that the map
$(w,r)\mapsto\big(\lambda(w,r),\mu(w,r)\big)$ is one-to-one on the set
$$S=\{(w,r)\in \ZZ\times\ZZ:0\leq w, \,0\le r\le\ell,\,0\leq  m+w+r\},$$
 which is not easy to see from the formulas. We will need the following generalization.

 \begin{prop}\label{Bezout}
 Let $(w,r)\in S$ and let $(d,s)\in \RR\times\RR$ with $d,s>-1$. If $\big(\lambda(w,r),\mu(w,r)\big)=\big(\lambda(d,s),\mu(d,s)\big)$, then $(w,r)=(d,s)$.
 \end{prop}
 \begin{proof}
 Let us consider the algebraic plane curves
 \begin{equation*}
 \begin{split}
 F(x,y)&=x(x+m+\ell+y+n)+y(m+y-k+n)+\lambda,\\
 G(x,y)&=x(m+\ell-y)(x+m+\ell+y+n)+y(m-k)(m+y-k+n)+\mu,
\end{split}
 \end{equation*}
 where we put $\lambda=\lambda(w,r)$ and $\mu=\mu(w,r)$. Then it is easy to check that $F$ is irreducible and that $F$ and $G$ have no common component.

 Now one can see that both curves meet at the following points:
 $$(w, r), (w,-r-w-n+k-m), (-w-m-r-\ell-n,r),$$
$$(-w-m-r-\ell-n, k+\ell+w), (-\ell+r-k,k+\ell+w), (-\ell+r-k,-r-w-n+k-m).$$
These are six different points because $(w,r)\in S$ and $1\le k\le n-1$. Therefore by Bezout's theorem, see \cite{F}, these are the only points of intersection of $F$ and $G$. But the only point of intersection with $x>-1$ and $y>-1$ is $(w,r)$. This completes the proof of the proposition. \qed
 \end{proof}

 \subsection{The orthogonal functions $F_{w,r}$}  \mbox{} \\
 The aim of this subsection is to prove that the functions $F=F_{w,r}$ associated to the spherical functions $\Phi^{\mm(w,r), \kk}$ are polynomial functions that are  orthogonal with respect to a certain inner product.

\smallskip
Let us equip $V_\kk$ with a $K$-invariant inner product. Then the
$L^2$-inner product of two continuous functions $\Phi_1$ and
$\Phi_2$ with values in $V_\kk$ is
$$
\langle\Phi_1,\Phi_2\rangle=\int_G\tr(\Phi_1(g)\Phi_2(g)^*)\,dg,
$$
where $dg$ denotes the normalized Haar measure of $G$ and
$\Phi_2(g)^*$ denotes the adjoint of $\Phi_2(g)$.

In particular if $\Phi_1$ and $\Phi_2$ are two irreducible spherical
functions of type $\kk$, we write as above $\Phi_1=H_1\Phi_\kk$ and
$\Phi_2=H_2\Phi_\kk$. Then by using the integral formula given in
Corollary 5.12, p. 191 in \cite{He} we obtain
\begin{equation}\label{Hilbert}
\langle\Phi_1,\Phi_2\rangle=\int_0^1H_2^*(u)V(u)H_1(u)\,du
\end{equation}
where $V(u)$ is the weight matrix
$$V(u)=2n\sum_{r=0}^\ell \textstyle\binom{\ell+k-r-1}{\ell-r}\binom{n-k+r-1}{r}
 (1-u)^{m+\ell-r}u^{n-1}E_{rr}.$$

In Theorem \ref{D} we introduced the function $\Psi$
\begin{equation*}\label{Psi}
\Psi(u)=XT(u),\quad  \quad T(u)=\sum_{i=0}^\ell u^i E_{ii} \quad\text{ and } \quad  X=\sum_{i,j} \textstyle\binom{i}{j}E_{ij},
\end{equation*}
to conjugate the differential operator $D$ to a hypergeometric  operator $\tilde D$.

Therefore in terms of the functions $F_1=\Psi^{-1}H_1$ and
$F_2=\Psi^{-1}H_2$ we have
\begin{equation*}
\langle F_1,F_2\rangle_W=\int_0^1 F_2(u)^*W(u)F_1(u)\,du,
\end{equation*}
where the weight matrix $W(u)=\Psi^*(u)V(u)\Psi(u)$ is given by
\begin{equation*}\label{peso}
W(u)=\sum_{i,j=0}^\ell\left( \sum_{r=0}^\ell \textstyle \binom ri
\binom rj \binom{\ell+k-r-1}{\ell-r}\binom{n-k+r-1}{r}
 (1-u)^{m+\ell-r}u^{i+j+n-1}\right) E_{ij}.
 \end{equation*}

We point out that the weight function $W(u)$ has finite moments of
all orders if and only if $m\ge0$. See \cite{PT2} for some details.

\begin{prop}\label{ortogonalidad}
  The functions $F_{w,r}$ associated to irreducible spherical
functions $\Phi^{\mm(w,r), \kk}$ of type $\kk$ are orthogonal with respect to the weight
function $W(u)$, that is
\begin{equation*}
\langle F_{w,r}\,,F_{w',r'}\rangle_W=0,\quad\text{if}\quad
(w,r)\ne(w',r').
\end{equation*}
\end{prop}
\begin{proof} The differential operators $\Delta_P$ and $\Delta_Q$ are
symmetric with respect to the $L^2$-inner product, among matrix-valued functions on $G$. Therefore the differential operators
$D$ and $E$ are symmetric with respect to the inner product
\eqref{Hilbert}, among vector-valued functions on the interval
$[0,1]$. This implies that $\tilde D$ and $\tilde E$ are symmetric with respect to the inner product defined by the weight matrix  $W(u)$.

Since  the pairs $(\lambda,\mu)$ of eigenvalues of the symmetric
operators $\tilde D$ and $\tilde E$ characterize these $F$'s, see  Theorem \ref{caracterizacion}, it follows that
the $F$'s associated to irreducible spherical
functions of type $\kk$ are orthogonal with respect to this inner product.
\qed
\end{proof}

\medskip
Let $L^2_W$ denote the Hilbert space of all $\CC^{\ell+1}$-valued
functions on $[0,1]$, squared integrable with respect to the inner product \eqref{Hilbert}. Also let $\CC^{\ell+1}[u]$ be
the $\CC^{\ell+1}$-valued polynomial  functions.

Let us consider the following linear space,
\begin{equation}\label{V}
\mathcal{V}=\{P\in\CC^{\ell+1}[u]:(\Psi P)_s=(1-u)^{s-m-\ell}g_s,\,\, g_s\in\CC[u], m+\ell+1\leq s\leq \ell\}.
\end{equation}

\noindent Observe that $\mathcal{V}=\CC^{\ell+1}[u]$ if and only if $m\ge0$.

\begin{lem}
 $\mathcal{V}$ is a linear subspace of $L^2_W.$
 Moreover $\mathcal V$ is stable under the differential operator $D$.
\end{lem}
\begin{proof}
Let $P\in\mathcal{V}$.
Then
\begin{align*}
\|P\|^2_W=&\int_0^1 P^*\Psi^*V(u)\Psi P\,du=\int_0^1(\Psi P)^*V(u)(\Psi P)\,du \\
=&\sum_{s=0}^\ell \int_0^1 c_s(\Psi P)^*_s(\Psi P)_su^{n-1}(1-u)^{m+\ell-s}\,du
\end{align*}
where $c_s=2n\binom{\ell+k-s-1}{\ell-s}\binom{n-k+s-1}{s}$.
\begin{align*}
\|P\|^2_W=&\sum_{s=0}^{m+\ell} \int_0^1 c_s |(\Psi P)_s|^2u^{n-1}(1-u)^{m+\ell-s}\,du\\
&+\sum_{s=m+\ell+1}^\ell \int_0^1 c_s u^{n-1}(1-u)^{s-m-\ell}|g_s(u)|^2\,du<\infty,
\end{align*}
because $\Psi $, $P$ and $g_s$ are polynomial functions.  Therefore $P\in L^2_W$. \\
The proof of the last statement is left to the interested reader.
\qed
\end{proof}

\bigskip
Let us consider the closure $\bar{\mathcal V}$  of the
subspace ${\mathcal V}$ in $L^2_W$.

\begin{prop}\label{Fenclausura}
Let $F$ be the $\CC^{\ell+1}$-valued function associated to an
irreducible spherical function $\Phi$ of type $\kk$. Then
$F\in\bar{\mathcal{V}}$.
\end{prop}
\begin{proof}
The function $F=\Psi^{-1}H$ is analytic at $u=1$ because   $H$ and $\Psi^{-1}$
are analytic at $u=1$. Hence
$$F=\sum_{j=0}^\infty (1-u)^jF_j,\qquad \text{ } F_j\in \CC^{\ell+1}.$$

It is enough to prove that the partial sums $\sum_{j=0}^N(1-u)^jF_j$ are in $\mathcal{V}$, because $F$ is analytic on $[0,1]$. In other words,
if $m+\ell+1\leq s\leq \ell$ we need to show that
$\big(\Psi\sum_{j=0}^N(1-u)^jF_j\big)_s=(1-u)^{s-m-\ell} g_s$, for some polynomial $g_s$.

We can write $\Psi(u)=\sum_{0\leq k \leq
\ell}(1-u)^k\Psi_k,$ where the coefficients $\Psi_k$ are
$(\ell+1)\times(\ell+1)$ matrices. Then
\begin{equation*}
\begin{split}
(\Psi F)_s&=\sum_{j=0}^\infty (1-u)^j \sum_{k=0}^\ell (1-u)^k(\Psi_k F_j)_s =\sum_{r=0}^\infty (1-u)^{r}
\sum_{k=0}^{\min\{r,\ell\}}(\Psi_kF_{r-k})_s.
\end{split}
\end{equation*}
From Theorem \ref{asintotica} we have that $h_s= (\Psi F)_s$ is analytic in $[0,1]$ and it has a zero of order at least $s-m-\ell$ at $u=1$. Then we have
$$\sum_{k=0}^{\min\{r,\ell\}}(\Psi_k F_{r-k})_s=0,\text{\;\; for all}\;\;r<s-m-\ell.$$
Therefore
$$\Big(\Psi\sum_{j=0}^N(1-u)^jF_j\Big)_s=\sum_{r=s-m-\ell}^{N+\ell} (1-u)^{r}
 \sum_{k=0}^{\min\{r,\ell\}} (\Psi_k F_{r-k})_s,$$
and thus $\sum_{j=0}^N(1-u)^jF_j\in\mathcal{V}$. This completes the proof.
 \qed
\end{proof}

\begin{thm}\label{esfpolinom}
The function $F_{w,r}$ associated to the spherical function
$\Phi^{\mm(w,r),\kk}$ is a polynomial function.
\end{thm}
\begin{proof} Set
$${\mathcal V}_j=\{P\in {\mathcal V}: \deg P\le j\}.$$
Then ${\mathcal V}_j$ is $D$-stable. Since $D$ is symmetric ${\mathcal V}_j^\perp\cap{\mathcal V}_{j+1}$
is invariant under $D$, and there exists an orthonormal basis of
${\mathcal V}_j^\perp\cap{\mathcal V}_{j+1}$ of eigenvectors of $D$.
Then by induction on $j\ge0$ it follows that there exists an
orthonormal basis $\{P_i\}$ of ${\mathcal V}$ such that
$DP_i=\lambda_iP_i$, $\lambda_i\in\CC$.

By Proposition \ref{Fenclausura} $F\in\bar{\mathcal{V}}$. Therefore
$$F=\sum_{j=0}^\infty a_i P_i,\qquad a_i=\langle F,P_i\rangle_W.$$

Since $F$ is analytic on $[0,1]$ the series converges not only in
$L^2_W$ but also in the topology based on uniform convergence of
sequences of functions and their successive derivatives on compact
subsets of $(0,1)$. Therefore
$$\lambda F=D F=D\Big(\sum_{j=0}^\infty a_i
P_i\Big)=\sum_{i=0}^\infty a_i\lambda_i P_i,$$ and thus $a_i=0$
unless $\lambda=\lambda_i$. Then
$$F=\sum_{\lambda=\lambda_i}a_i P_i,$$
and since $\dim V_\lambda<\infty$ we conclude that the function $F$
is a polynomial. \qed
\end{proof}

\begin{prop} \label{leading}
The function $F_{w,r}$  is a polynomial function of degree $w$ and
its leading coefficient is of the form
$F_w=(x_0,\dots,x_r,0,\dots,0)^t$ with $x_r\ne0$.
\end{prop}

\begin{proof}
From Theorem \ref{esfpolinom} we have that $F=F_{w,r}$ is a
polynomial function; let us say that it is of degree $d$ and leading
coefficient $F_d$. From Theorem \ref{caracterizacion} we know that
$F$ is an eigenfunction of the differential operators $D$ and $E$,
namely
$$ DF=\lambda(w,r)F,  \qquad EF=\mu(w,r)F.$$
By looking at the leading coefficient of the polynomials in the above identities we have that
\begin{align*}
  & \left(d(d-1+U)+V+\lambda(w,r) \right)F_d=0,\\
  & \left(d(d-1)M_1+dP_1+(m-k)V+\mu(w,r) \right)F_d=0.
\end{align*}
Let us introduce the matrices
\begin{align*}
  L_1 &= d(d-1+U)+V+\lambda(w,r) \\
  &= \sum_{s=0}^\ell (\lambda(w,r)-\lambda(d,s)) E_{ss} + \sum_{s=0}^{\ell-1}(\ell-s)(n+s-k) E_{s,s+1},
\\
  L_2 &= d(d-1)M_1+dP_1+(m-k)V+\mu(w,r)  \\
  &= \sum_{s=0}^\ell (\mu(w,r)-\mu(d,s)) E_{ss} + \sum_{s=0}^{\ell-1}(\ell-s)(n+s-k) (d-m+k)E_{s,s+1}.
\end{align*}

Now we observe that $L_1$ has a nontrivial kernel if and only if $$\lambda(w,r)=\lambda(d,j) \quad \text{ for some } \quad 0\leq j\leq \ell.$$
In this case, if $s=\min\{j:\lambda(w,r)=\lambda(d,j)\}$ and $0\neq{\bf x}=(x_0, \dots, x_\ell)\in \ker(L_1)$, then $x_{s}\neq 0$ and
$x_j=0$ for $s+1\leq j\leq \ell$; this is so because the elements above te main diagonal of $L_1$ are not zero.

Similarly, $L_2$ has a nontrivial kernel if and only if  $\mu(w,r)=\mu(d,j)$ for some $0\leq j\leq \ell$.
Now the elements above the main diagonal of $L_2$ are all zero or all are not zero.
In the first case since $F_d\in \ker(L_1)\cap\ker(L_2)$ we must have $\mu(w,r)=\mu(d,s)$.
In the second case if $s'=\min\{j:\mu(w,r)=\mu(d,j)\}$ and $0\neq{\bf x}=(x_0, \dots, x_\ell)\in \ker(L_2)$,
then $x_{s'}\neq 0$ and $x_j=0$ for $s'+1\leq j\leq \ell$. Since $F_d\in \ker(L_1)\cap\ker(L_2)$ we must have $s=s'$.
In any case  we have $\lambda(w,r)=\lambda(d,s)$ and  $\mu(w,r)=\mu(d,s)$.

From Proposition \ref{Bezout}  we have that $(w,r)=(d,s)$. In
particular $d=\text{deg}(F_{w,r})=w$ and the leading coefficient is
of the required form. \qed
\end{proof}

\section{Matrix  orthogonal polynomials}

\noindent From now on we shall assume that $m\ge0$.

In this section we package appropriately  the functions $F_{w,r}$ associated to the spherical functions
$\Phi^{\mm(w,r), \kk}$ of one step $K$-type, for $0\leq w$ and $0\leq r \leq \ell$,  in order to obtain  families of examples of matrix-valued orthogonal
polynomials that are eigenfunctions of certain second-order differential operators.

It is important to remark that these examples arise naturally from the
group representation theory, in particular from  spherical functions.

\smallskip
The theory of matrix-valued orthogonal polynomials, without any consideration of differential equations goes back to M. G. Krein, in \cite{K1} and \cite{K2}.
Starting with a selfadjoint positive definite weight matrix  $W=W(u)$, with finite moments we can consider the skew-symmetric bilinear form defined
for matrix-valued polynomials $P$ and $Q$ by
$$\langle P,Q\rangle_W=\int_\RR P(u)W(u)Q(u)^*\,du=0.$$
By the usual construction this leads to the existence of  sequences $(P_w)_{w\geq 0}$ of matrix valued orthogonal polynomials
with non singular leading coefficients and $\deg  P_w=w$.

Any  sequence of matrix orthogonal polynomials $(P_w)_w$
satisfies a three-term recursion relation of the form
\begin{equation}\label{eq1}
uP_w(u) = A_wP_{w-1}(u) + B_wP_w(u) + C_wP_{w+1}(u),\quad w \ge 0,
\end{equation}
where we put $P_{-1}(u) = 0$.

In \cite{D} the study of matrix-valued orthogonal polynomials that are eigenfunctions of certain
second-order symmetric differential operator
was started.
A differential operator with matrix
coefficients acting on matrix-valued functions could be made to act  either
on the left or on the right. One finds a discussion of these two
actions in \cite{D}. The conclusion there is that if one wants to
have matrix weights $W(u)$ that are not direct sums of scalar ones and
that have matrix polynomials as their eigenfunctions, one should
settle for right-hand side differential operators. We agree now that $D$ given by
$$D=\sum_{i=0}^s \partial^iF_i(u),\quad\quad \partial=\frac{d}{du},$$
acts on $P(u)$ by means of
$$ P D = \sum_{i=0}^s \partial^i(P)(u)F_i(u).$$
One could make $D$ act on $P$ on the right as defined above, and
still write down the symbol $DP$ for the result. The advantage of
using the notation $PD$ is that it respects associativity: if $D_1$
and $D_2$ are two differential operators, we have
$P(D_1D_2)=(PD_1)D_2$.

\

We  define the matrix polynomial $P_w$ as the $(\ell+1)\times(\ell+1)$ matrix whose
$r$-row is the polynomial $F_{w,r} (u)$, for $0\leq r \leq \ell$. In other words
\begin{equation}\label{presidente}
P_w=(F_{w,0},\dots,F_{w,\ell})^t.
\end{equation}

An explicit expression for the rows of $P_w$ is given in Theorem \ref{caracterizacion},
in terms of the matrix hypergeometric function,
namely
\begin{equation*}\label{final3}
F_{w,r}(u)={}_2\!H_1 \!\left( \begin{smallmatrix} U\,;\,V+\lambda(w,r)\\
U-C\end{smallmatrix} ; u\right)F_{w,r}(0),
\end{equation*}
where its value $F_{w,r}(0)$ at $u=0$ is the $\mu(w,r)$-eigenvector
of $M(\lambda(w,r))$ properly normalized.

We notice that  from  Proposition \ref{ortogonalidad} the rows of $P_w$ are orthogonal with respect to the inner product $\langle \cdot,\cdot\rangle_W$;
then the sequence $(P_w)_w$ is orthogonal  with respect to the weight matrix $W(u)$,
$$\langle P_w,P_{w'}\rangle_W=\int_0^1P_w(u)W(u)P_{w'}(u)^*\,du=0, \quad \text{ for all $w\ne w'$}.$$

From Proposition \ref{esfpolinom} we obtain that each row of $P_w$ is a polynomial function of degree $w$.
A more careful look at the definition  shows that $P_w$ is a matrix polynomial  whose  leading coefficient is
a lower triangular nonsingular  matrix, see Proposition \ref{leading}. In other words this sequence of matrix-valued polynomials
fits squarely within Krein's theory and we obtain the following proposition.

\begin{prop}
  The matrix polynomial functions $P_w$,  whose $r$-row is the polynomial $F_{w,r}$, $0\leq r\leq \ell$, form an sequence of orthogonal
  matrix polynomials with respect to $W$.
\end{prop}

\smallskip
In \cite{D} a sequence of matrix orthogonal polynomials is called classical if there exists a symmetric differential operator of
order-two that has these polynomials as eigenfunctions with a matrix eigenvalue.

In the next proposition we show that the $(P_w)_w$ is a family of classical orthogonal polynomials
featuring two algebraically independent symmetric differential operators of order two.

\begin{prop} The polynomial function $P_w$ is an eigenfunction of the differential operator $D^t$ and $E^t$, the transposes of the operators
$D$ and $E$ appearing in \eqref{DD,EE}. Moreover
$$ P_wD^t=\Lambda_w(D^t)P_w,\qquad P_wE^t=\Lambda_w(E^t)P_w,$$
  where
$\Lambda_w(D^t)=\text{diag}(\lambda(w,0),\dots,\lambda(w,\ell))$ and
$\Lambda_w(E^t)=\text{diag}(\mu(w,0),\dots,\mu(w,\ell))$.
\end{prop}

\begin{proof}
We have
\begin{align*}
  D^t&=\partial^2u(1-u)+\partial(U^t-C^t-uU^t)-V^t,\\
E^t&=\partial^2(1-u)(M_0^t-M_1^t+uM_1^t)+\partial(P_1^t-P_0^t-uP_1^t)-(m-k)V^t,
\end{align*}
where the coefficient matrices are  given explicitly in Propositions \ref{D} and \ref{E}.

\noindent Now we observe that
\begin{align*}
P_wD^t&=\;u(1-u)P''_w+P'_w(U^t-C^t-uU^t)-P_wV^t\\
&=u(1-u)(F''_{w,0},\dots,F''_{w,\ell})^t+(F'_{w,0},\dots,F'_{w,\ell})^t(U-C-uU)^t\\
&\qquad -(F_{w,0},\dots,F_{w,\ell})^tV^t\\
&=\;(DF_{w,0},\dots,DF_{w,\ell})^t=(\lambda(w,0)F_{w,0},\dots,\lambda(w,\ell)F_{w,\ell})^t \\
&=\Lambda_w(D^t)P_w,
\end{align*}
where
$\Lambda_w(D^t)=\text{diag}(\lambda(w,0),\dots,\lambda(w,\ell))$.

Similarly $P_wE^t=\Lambda_w(E^t)P_w$ where
$\Lambda_w(E^t)=\text{diag}(\mu(w,0),\dots,\mu(w,\ell))$ and this concludes the proof. \qed
\end{proof}

\smallskip

In general given a matrix weight $W(u)$ and a sequence of matrix orthogonal polynomials $P_w$ one can  dispense of the requirement of  symmetry and
consider the algebra of all matrix  differential operators $D$ such that $$P_w D = \Lambda_w(D) P_w,$$
where each $\Lambda_w$ is a matrix. See Section 4 of \cite{GT}.
The set of these $D$ is denoted by ${\mathcal D}(W)$.

Starting with \cite{GPT4}, \cite{GPT2}, \cite{G1} and \cite{DG} one has a
growing collection of weight matrices $W(u)$ for which the algebra
$\mathcal D(W)$ is not trivial, i.e., does not consist only of
scalar multiples of the identity operator.
     A first attempt to go beyond the issue of the existence of one
nontrivial element in $\mathcal D(W)$ was undertaken in \cite{CG}, for some examples  whose weights  are  of size two. The study
of the full algebra, in one of this cases was considered in \cite{T3}.

\subsection{The three-term recursion relation satisfied by the polynomials $P_w$}
\mbox{}

\noindent The aim of this subsection is to indicate how to obtain, from the representation theory, the three term recursion relation  satisfied
by the sequence of matrix orthogonal polynomials $(P_w)_w$ built up from packages of spherical functions of the pair $(G,K)$.
More details can be found in \cite{PT1} for the case $n=2$ and in \cite{P2} for the general case.

\smallskip
The standard representation of  $\U(n+1)$ on  $\CC^{n+1}$ is  irreducible and its
highest weight is $(1,0,\dots , 0)$. Similarly the representation of  $\U(n+1)$ on the dual of
$\CC^{n+1}$ is  irreducible and its highest weight is $(0,\dots,0 , -1)$. Therefore we have that
$$\CC^{n+1}=V_{(1,0,\cdots, 0)}\quad \text{ and }\quad (\CC^n)^*=V_{(0,\dots,0,-1)}.$$

For any irreducible representation of  $\U(n+1)$ of highest weight  $\mm$,   the tensor product
$V_\mm\otimes \CC^{n+1}$  decomposes as a direct sum of $\U(n+1)$-irreducible representations in the following way:
\begin{equation}\label{tensor1}
  V_\mm\otimes \CC^{n+1}\simeq V_{\mm+\ee_1}\oplus V_{\mm+\ee_2}\oplus \cdots \oplus V_{\mm+\ee_{n+1}},
\end{equation}
and
\begin{equation}\label{tensor2}
  V_\mm\otimes (\CC^{n+1})^*\simeq V_{\mm-\ee_1}\oplus V_{\mm-\ee_2}\oplus \cdots \oplus V_{\mm-\ee_{n+1}},
\end{equation}
where $\{\ee_1,\cdots,\ee_{n+1}\}$ is the canonical basis of $\CC^{n+1}$, see \cite{VK}.
\begin{remark} The irreducible modules on the right-hand side of \eqref{tensor1} and \eqref{tensor2} whose
parameters $(m'_1,m'_2,\dots ,m'_{n+1})$ do not
satisfy the conditions $m_1'\geq m_2'\geq \dots \geq m_{n+1}'$ have to be omitted.
\end{remark}

Let $\kk=(k_1,\dots , k_n)$ be the highest weight of an irreducible representation of $\U(n)$ contained in the representation
 $\mm=(m_1,\dots , m_{n+1})$ of $\U(n+1)$.
From \eqref{tensor1} and \eqref{tensor2}, the following multiplication formulas are proved in \cite{P2}.

\begin{prop}\label{multiplicationformula}
 Let $\phi$ and $\psi$ be, respectively, the one-dimensional spherical functions associated to the
 standard representation of $G$ and its dual. Then
\begin{align*}
\phi(g)\Phi^{\bf m,\kk}(g)&=\sum_{i=1}^{n+1}a_i^2(\mm, \kk) \Phi^{{\mathbf m}+{\mathbf
e}_i, \kk}(g),\\
\psi(g)\Phi^{\bf m, \kk}(g)&=\sum_{i=1}^{n+1}b_i^2(\mm, \kk) \Phi^{{\bf m}-{\bf
e}_i,\kk}(g).
\end{align*}
The constants $a_i(\mm, \kk)$ and $b_i(\mm, \kk)$  are given by
\begin{equation}\label{coefficientsab}
\begin{split}
a_i(\mm,\kk)&=\left\vert\frac{\prod_{j=1}^{n}(k_j-m_i-j+i-1)}
{\prod_{j\ne i}(m_j-m_i-j+i)}\right\vert^{1/2},\\
b_i(\mm,\kk)&=\left\vert\frac{\prod_{j=1}^{n}(k_j-m_i-j+i)}
{\prod_{j\ne i}(m_j-m_i-j+i)}\right\vert^{1/2}.
\end{split}
\end{equation}
Moreover
\begin{equation}\label{sumauno}
\sum_{i=1}^{n+1}a_i^2(\mm,\kk)=\sum_{i=1}^{n+1}b_i^2(\mm,\kk)=1.
\end{equation}
\end{prop}

By combining both multiplication formulas given in the previous theorem we obtain the following identity involving different spherical
functions of the same $K$-type.

\begin{cor}\label{ntermrecursion}
$$\phi(g)\psi(g)\Phi^{\bf m,\kk}(g)=\sum_{i,j=1}^{n+1}a_j^2(\mm, \kk)b_i^2(\mm+\ee_j, \kk) \Phi^{{\mathbf m}+{\mathbf
e}_j-\ee_i, \kk}(g).$$
\end{cor}

In this section we are restricting our attention to spherical functions of
type
$$ \kk=(\underbrace{m+\ell,\dots,m+\ell}_k,\underbrace{ m,\dots,m}_{n-k}),$$
with $m\geq 0$. Thus we only consider $(n+1)$-tuples $\mm$ of the  form
\begin{equation*}\label{mwr}
\mm(w,r)=(w+m+\ell,\underbrace{m+\ell,\dots,m+
\ell}_{k-1},m+r,\underbrace{m, \dots,m}_{n-k-1},-w-r),
\end{equation*}
for $0\le w$ and $0\le r\le\ell$.

Now we introduce the $rs$-entries of  the matrices $A_w, B_w$ and $C_w$:
\begin{align*}
     (A_w)_{rs} & =\begin{cases}
       a^2_{n+1}(\mm(w,r))b^2_{1}(\mm(w,r)+\ee_{n+1}) & \quad\text{ if } s=r\\
       a^2_{k+1}(\mm(w,r))b^2_{1}(\mm(w,r)+\ee_{k+1})  & \quad \text{ if } s=r+1\\
       0 & \quad  \text{ otherwise}
     \end{cases} \displaybreak[0]
\\
     (C_w)_{rs} & =\begin{cases}
       a^2_{1}(\mm(w,r))b^2_{d+1}(\mm(w,r)+\ee_{1})) & \quad\text{ if } s=r\\
       a^2_1(\mm(w,r))b^2_{k+1}(\mm(w,r)+\ee_1)  & \quad \text{ if } s=r-1\\
       0 & \quad  \text{ otherwise}
     \end{cases} \displaybreak[0]
\\
     (B_w)_{rs} & =\begin{cases}
      \displaystyle \sum_{1\leq j\leq n+1}a^2_j(\mm(w,r))b^2_j(\mm(w,r)+\ee_j)) & \quad\text{ if } s=r\\
       a^2_{k+1}(\mm(w,r))b^2_{n+1}(\mm(w,r)+\ee_{k+1})  & \quad \text{ if } s=r+1\\
       a^2_{n+1}(\mm(w,r))b^2_{k+1}(\mm(w,r)+\ee_{n+1}) & \quad\text{ if } s=r-1\\
       0 & \quad  \text{ otherwise}
     \end{cases}
  \end{align*}
where $a_i^2(\mm(w,\rr))=a_i^2(\mm(w,r),\kk)$,
$b_i^2(\mm(w,r)+\ee_j)=b_i^2((\mm(w,r)+\ee_j,\kk))$ for $1\le j\le
n+1$, see \eqref {coefficientsab}.

Notice that $A_w$ is an upper-bidiagonal matrix, $C_w$ is a lower-bidiagonal matrix and $B_w$ is a tridiagonal matrix.

\smallskip
For the benefit of the reader we
include the following simplified expressions, in the case of one-step $K$-type, of the coefficients involved in the definition of the matrices
$A_w,B_w$ and $C_w$:
\begin{align*}
  a_1^2(\mm(w,r))&= \frac{(w+k)(w+\ell+n)}{(w+\ell-r+k)(2w+m+n+\ell+r)},\\
  a_{k+1}^2(\mm(w,r))&= \frac{(\ell-r)(r+n-k)}{(w+\ell-r+k)(w+m+n+2r-k)},\\
  a_{n+1}^2(\mm(w,r))&=\frac{(w+m+n+\ell+r-k)(w+m+r)}{(w+m+n+2r-k)(2w+m+n+\ell+r)};
\end{align*}
all the others $a_j^2(\mm(w,r))$ are zero. The remaining coefficients are

\begin{align*}
  b_1^2(\mm(w,r)+\ee_1)&= \frac{(w+1)(w+\ell+k+1)}{(w+\ell-r+k+1)(2w+m+n+\ell+r+1)},\displaybreak[0]\\
  b_1^2(\mm(w,r)+\ee_{k+1})&= \frac{w(w+\ell+k)}{(w+\ell-r+k-1)(2w+m+n+\ell+r)},\displaybreak[0]\\
  b_1^2(\mm(w,r)+\ee_{n+1})&=\frac{w(w+\ell+k)}{(w+\ell-r+k)(2w+m+n+\ell+r-1)},
\end{align*}

\smallskip
\begin{align*}
  b_{k+1}^2(\mm(w,r)+\ee_1)&= \frac{r(\ell-r+k)}{(w+\ell-r+k+1)(w+m+n+2r-k)},\displaybreak[0]\\
  b_{k+1}^2(\mm(w,r)+\ee_{k+1})&= \frac{(r+1)(\ell-r+k-1)}{(w+\ell-r+k-1)(w+m+n+2r-k+1)},\displaybreak[0]\\
  b_{k+1}^2(\mm(w,r)+\ee_{n+1})&=\frac{r(\ell-r+k)}{(w+\ell-r+k)(w+m+n+2r-k-1)},
\end{align*}

\smallskip
\begin{align*}
  b_{n+1}^2(\mm(w,r)+\ee_1)&= \frac{(w+m+n+\ell+r)(w+m+n+r-k)}{(w+m+n+2r-k)(2w+m+n+\ell+r+1)},\displaybreak[0]\\
  b_{n+1}^2(\mm(w,r)+\ee_{k+1})&= \frac{(w+m+n+\ell+r)(w+m+n+r-k)}{(w+m+n+2r-k+1)(2w+m+n+\ell+r)},\displaybreak[0]\\
  b_{n+1}^2(\mm(w,r)+\ee_{n+1})&=\frac{(w+m+n+\ell+r-1)(w+m+n+r-k-1)}{(w+m+n+2r-k-1)(2w+m+n+\ell+r-1)}.
\end{align*}

\medskip
In the open subset $\{a(\theta)\in A:0<\theta<\pi/2\}$ of $A$,
we introduce the coordinate $t=\cos^2(\theta)$ and define on the
open interval $(0,1)$ the matrix-valued function
$$\Phi(w,t)=\big(\Phi^{\mm(w,r),\kk}_{s}(a(\theta))\big)_{0\leq r,s\leq \ell}.$$

\begin{prop} \label{multiplicationformula1} For   $0<t<1$  and  all integer $0\le w$ we have
\begin{equation*}
t\Phi(w,t)= A_w\Phi(w-1,t)+ B_w\Phi(w,t)+ C_w\Phi(w+1,t).
\end{equation*}
where we put $\Phi(-1,t) = 0$.
\end{prop}

\begin{proof}  This result is a consequence of Corollary \ref{ntermrecursion}
and of the definitions of the matrices  $A_w, B_w, C_w$, when we take $g=a(\theta)$.

 We recall that  $\phi(g)$ and $\psi(g)$ are the one-dimensional spherical
 functions associated to the $G$-modules $\CC^{n+1}$ and $(\CC^{n+1})^*$, respectively.
 A direct computation gives
$$\phi(a(\theta))=\langle a(\theta)e_{n+1}, e_{n+1}\rangle =\cos \theta \quad \text{ and } \quad \psi(a(\theta))=\langle a(\theta)\lambda_{n+1},
\lambda_{n+1}\rangle =\cos \theta.$$ Then
$\phi(a(\theta))\psi(a(\theta))=\cos^2(\theta)=t$. \qed
\end{proof}

\medskip
Now we take into account that  the polynomial functions  $P_{w}$ introduced in \eqref{presidente}, are obtained by right-multiplying
the function   $\Phi(w,t)$ by a matrix function on $t$, independent of $w$.  After  the change of variable $t=1-u$
 we obtain the following three term recursion relation for the polynomials $P_w$.

\begin{thm}
  The polynomial functions $P_w(u)$ satisfy
  \begin{equation*}\label{3term}
(1-u)P_w(u)= A_w P_{w-1}(u)+ B_w P_{w}(u)+ C_wP_{w+1}(u),
\end{equation*}
for all  $w \ge 0$, where the matrices $A_w$, $B_w$ and $C_w$ were introduced before.
\end{thm}

\smallskip
This three term recursion relation
can be written in the following way:
\begin{equation}\label{matrix}
(1-u)\begin{vmatrix}P_0\\P_1\\P_2\\P_3\\\cdot\end{vmatrix}
=\begin{vmatrix} B_0&C_0&0&\\A_1&B_1&C_1&0&\\0&A_2&B_2&C_2&0&\\ &0&A_3&B_3&C_3&0\\&&\cdot&\cdot&\cdot&\cdot&\cdot\end{vmatrix}
\begin{vmatrix}P_0\\P_1\\P_2\\P_3\\\cdot\end{vmatrix},
\end{equation}

From \eqref{sumauno} one can prove that
$$\sum_{s=0}^{\ell}\big((A_w)_{rs}+(B_w)_{rs}+(C_w)_{rs}\big)=1,$$
showing that the square semi-infinite matrix $M$ appearing in
\eqref{matrix} is an stochastic matrix.

\smallskip
 In \cite{GPT6} we describe a random mechanism
based on Young diagrams that gives rise to a random walk in the set
of all Young diagrams of $2n+1$ rows and whose $2j$-th row has $k_j$
boxes $1\le j\le n$, whereby in one unit of time one of the $m_i$ is
increased by one with probability $a_i^2(\mm,\kk)$ see
\eqref{coefficientsab}.

\medskip
In the expressions for $D,E,W,P_w$ and $M$ the discrete parameters
$(m,n)$ enter in a simple analytical fashion. Appealing to some
version of analytic continuation, it is clear that this entire
edifice remains valid if one allows $(m,n-1)$ to range over a
continuous set of real values  $(\alpha,\beta)$. The requirement
that $W$ retain the property of having finite moments  of all
orders translates into the conditions  $\alpha,\beta>-1$. We will
denote the corresponding weight and the orthogonal polynomials by
$W^{\beta,\alpha}$ and $P^{\beta,\alpha}_w$. These are some
interesting families of new matrix valued Jacobi polynomials.

\

\subsection*{Acknowledgments } \mbox{} \\ This paper was partially supported by CONICET, PIP 112-200801-01533 and by Secyt-UNC.

\



\end{document}